\documentclass[onefignum,onetabnum]{siamart190516}
\usepackage{amssymb}
\usepackage{amsmath}
\usepackage{amsfonts}
\usepackage{color}
\usepackage{graphicx}
\usepackage{bm}
\usepackage{mathtools}

\newcommand{\inc}{\bm{\iota}^{(m)}}
\newcommand{\R}{\mathbb{R}}
\newcommand{\N}{\mathbb{N}}
\newcommand{\Z}{\mathbb{Z}}
\newcommand{\cX}{\mathcal X}

\newcommand{\bx}{\bar{x}}
\newcommand{\ba}{\bar{a}}
\newcommand{\bb}{\bar{b}}

\newcommand{\tx}{\tilde{x}}
\newcommand{\bydef}{\,\stackrel{\mbox{\tiny\textnormal{\raisebox{0ex}[0ex][0ex]{def}}}}{=}\,}
\newcommand{\dagA}{A^\dagger}

\newsiamremark{remark}{Remark}

\headers{Free vibrations in a wave equation modeling MEMS}{C. Garc\'{\i}a-Azpeitia, and J.-P. Lessard}

\title{Free vibrations in a wave equation modeling MEMS\thanks{The second author was funded by NSERC.}}

\author{Carlos Garc\'{\i}a-Azpeitia\thanks{
Depto. Matem\'{a}ticas y Mec\'{a}nica
IIMAS, Universidad Nacional Aut\'{o}noma de M\'{e}xico, Apdo. Postal 20-726,
01000 Ciudad de M\'{e}xico, M\'{e}xico,
(\email{cgazpe@mym.iimas.unam.mx}).}
\and Jean-Philippe Lessard\thanks{McGill University, Department of Mathematics and
Statistics, 805 Sherbrooke Street West, Montreal, QC, H3A 0B9, Canada
  (\email{jp.lessard@mcgill.ca}), (\url{http://www.math.mcgill.ca/jplessard/}).}
}

\begin{document}

\maketitle

\begin{abstract}
We study a nonlinear wave equation appearing as a model for a membrane
(without viscous effects) under the presence of an electrostatic potential
with strength $\lambda$. The membrane has a unique stable branch of steady
states $u_{\lambda}$ for $\lambda\in\lbrack0,\lambda_{\ast}]$. We prove that
the branch $u_{\lambda}$ has an infinite number of branches of periodic
solutions (free vibrations) bifurcating when the parameter $\lambda$ is
varied. Furthermore, using a functional setting, we compute numerically the
branch $u_{\lambda}$ and their branches of periodic solutions. This approach
is useful to validate rigorously the steady states $u_{\lambda}$ at the critical value
$\lambda_{\ast}$.

Dedicated to the memory of G. Flores.

\end{abstract}


\section{Introduction}

We consider an idealized device that consists of an elastic plate suspended
above a rigid ground plate. This device falls in the category of
microelectromechanical systems (MEMS). The membrane is taken to be
rectangular with two fixed parallel sides, while the other sides are
considered to be thin and free. When a potential difference is applied
between the membrane and the plate, the membrane deflects towards the ground
plate. We assume that dissipation which might result from viscous effects on
the moving membrane can be neglected. Under these assumptions, the
deformation of the elastic membrane is described by the dimensionless
equation 
\begin{equation}
U_{tt}-U_{xx}+\frac{\lambda }{\left( 1+U\right) ^{2}}=0,\text{\qquad }x\in
\lbrack -\pi /2,\pi /2],  \label{1}
\end{equation}%
where $U(x,t)$ satisfies the Dirichlet boundary conditions $U(\pm \pi /2)=0$%
, and the parameter $\lambda $ represents the strength of the applied
voltage. A derivation of the nonlinearity leading to a general equation
modeling the electrostatic membrane%
\begin{equation*}
\varepsilon ^{2}U_{tt}+\nu U_{t}-U_{xx}+\frac{\lambda f(x)}{\left(
1+U\right) ^{2}}=0,
\end{equation*}
for which the nonlinear wave equation \eqref{1} is a particular case, can be
found in \cite{Pelesko_2002}, where the parameter $\varepsilon $ represents
the strength of the inertial term, $\nu $ the viscosity and the function $%
f(x)$ encodes the dielectric permitivity of the membrane.

In the design of microelectronic devices, it is relevant to study wether or
not the membrane touches the ground plate. This phenomenon is called \emph{%
touchdown} or \emph{quenching}. Mathematically, quenching occurs if there is
a point $(x,t)$ such that $U(x,t)=-1$. A vast literature exists on the study
of MEMS via parabolic and hyperbolic PDE modelling. Equation~\eqref{1} is in
fact a special case of the more general MEMS parabolic and hyperbolic PDE
models, and a vast mathematical literature is dedicated to their study. Let
us give a few examples. %

The local existence of solutions and the existence of quenching at a finite
time for a parabolic equation modelling MEMS is analyzed in \cite{Fl2,Gu1,Gu2}, and references therein. The case of a nonlocal parabolic
equation modelling MEMS is proposed in \cite{Fl3}. The existence of
solutions and the finite-time quenching for a damped wave equation modelling
MEMS is analyzed in \cite{Fl1}. Generalizations of the wave equation (\ref{1}%
) have been studied in \cite{Chang_Levine,Ka2,Smith}. 

The study of periodic orbits in MEMS models have also received their fair
share of attention. Their relevance comes from the fact that they persist as
small oscillations with no quenching. In \cite{Ka1}, periodic solutions were
observed numerically by solving an initial value problem for a non-local
wave equation modelling MEMS. The study of periodic solutions in Hamiltonian
PDEs (such as equation \eqref{1}) presents intrinsic problems associated to
infinite-dimensional kernels \cite{Ra78}, lack of compactness \cite{Ra}, or
small divisor problems \cite{CrWa93}. The small divisor problem was avoided
in \cite{AmZe80,Ki79,Ra78} by imposing restrictions on
the temporal period of the solutions of a nonlinear wave equation. By
imposing similar restrictions, the articles~\cite{Ki00} obtains the
existence of continuous branches of periodic solutions for a nonlinear wave
equation in a sphere, and in \cite{Ga19,Ga17} for a Hamiltonian
PDE appearing in the $n$-vortex filament problem.

The equation (\ref{1}) has a family of stable steady states $u_{\lambda}$
for $\lambda\in [0,\lambda_{\ast}]$. For instance, the existence of steady
states of \eqref{1} in the $N$-dimension ball was proven in \cite{Go1}. In
the present work, we prove the existence of continuous families of periodic
solutions near the branch of steady states (Theorem \ref{thm:main}). This is
our main contribution, and to the best of our knowledge, this result is new.
The main challenge encountered when proving the existence of the periodic
solutions is that the trivial branch $u_{\lambda}$ and its associated
spectrum are not known explicitly. Indeed, while the existence of continuous
families of periodic solutions has been obtained before for Hamiltonian PDEs
in \cite{Ga19,Ga17,Ki00,Ki79}, in those articles
the trivial branch and the spectrum of its linearized equation are known
explicitly. We overcome this problem with delicate estimates of the spectrum
which depend on an accurate estimate for the steady state $%
u_{\lambda_{\ast}} $ at the critical value $\lambda_{\ast}$. 
While the critical value $\lambda_{\ast}$ is known (see for instance \cite%
{Fl3,Ka2}), no estimate for the steady state $u_{\lambda_{\ast}}$ is known.
That leads to our second main contribution, which is to obtain precise and
rigorous estimate on the steady state $u_{\lambda_{\ast}}$ at the critical
value $\lambda_{\ast}$ (Theorem \ref{thm:saddle-node}). Finally, our third
contribution is to present a systematic approach to compute numerically the
families of periodic solutions using Chebyshev series expansion (in space)
and Fourier series expansions (in time) (see Figures \ref%
{fig:branches_po_k_1_q11},~\ref{fig:po_k1_q_11},~\ref%
{fig:branches_po_k_2_q47} and \ref{fig:po_k2_q_47}).

Specifically, the linear operator of the stationary equation (\ref{1}) at $%
u_{\lambda }$ is given by%
\begin{equation}
A(\lambda )u\bydef-\partial _{x}^{2}u-\frac{2\lambda }{\left( 1+u_{\lambda
}\right) ^{3}}u:H_{0}^{2}\subset L^{2}\rightarrow L^{2}\text{,}  \label{A}
\end{equation}%
where $H_{0}^{2}([-\pi /2,\pi /2];\mathbb{R})$ is the Sobolev space of
functions $u(x)$ satisfying Dirichlet boundary conditions $u(\pm \pi /2)=0$.
The operator $A$ is self-adjoint and positive definite for $\lambda \in
\lbrack 0,\lambda _{\ast })$ because the first eigenvalue $\mu _{1}(\lambda )
$ is positive (Theorem 4.2 in \cite{Go1}). Thus $A$ has eigenvalues $0<\mu
_{1}(\lambda )<\mu _{2}(\lambda )<\cdots $ with eigenfunctions satisfying%
\begin{equation}
A(\lambda )v_{k}(x;\lambda )=\mu _{k}(\lambda )v_{k}(x;\lambda ),\qquad k\in 
\mathbb{N}.  \label{ElipEigen}
\end{equation}%
Our main theorem regarding the existence of periodic solutions is the
following.

\begin{theorem}
\label{thm:main} There is an infinite number of non-resonant parameters $%
\lambda _{0}\in (0,\lambda _{\ast })$, associated with numbers $p,q,k\in 
\mathbb{N}$ by the relation $\mu _{k}(\lambda _{0})=\left( p/q\right) ^{2},$
such that there is a local continuum of $2\pi q/p$-periodic solutions
bifurcating from the steady state $u_{\lambda }(x)$ with $\lambda =\lambda
_{0}$. The local bifurcation consists of free vibrations satisfying the
estimates%
\begin{align}
U(t,x;\lambda )& =u_{\lambda _{0}}(x)+b\cos (pt/q)v_{k}(x;\lambda _{0})+%
\mathcal{O}_{C^{4}}(b^{2})\text{,}  \label{es} \\
\lambda & =\lambda _{0}+\mathcal{O}(b),  \notag
\end{align}%
where $b\in \lbrack 0,b_{0}]$ represents a parametrization of the local
bifurcation for some $b_{0}>0$ and $\mathcal{O}_{C^{4}}(b^{2})$ is a
function of order $b^{2}$ in the $C^{4}$-norm. Furthermore, the bifurcation
has symmetries%
\begin{align*}
U(t,x)& =U(-t,x)=U(t,-x),\hspace{1.88cm}\text{ for }k\text{\ odd,} \\
U(t,x)& =U(-t,x)=U(t+p\pi /q,-x),\qquad \text{ for }k\text{ even.}
\end{align*}
\end{theorem}

The existence of a branch of periodic solutions arising from the family of
steady states $u_{\lambda }$ is set as a branch of zeros for the functional
equation 
\begin{equation}
L(\lambda )u+g(u;\lambda )=0,\qquad L(\lambda )\bydef\left( p/q\right)
^{2}\partial _{t}^{2}+A(\lambda ),  \label{op}
\end{equation}%
where $A(\lambda )$ is given in \eqref{A} and $g=\mathcal{O}(u^{2})$ is an
analytic nonlinearity defined in a neighborhood of zero. Here $u\in
H_{sym}^{s}$ represents a perturbation from the steady state $u_{\lambda }$,
where $H_{sym}^{s}(S^{1}\times \lbrack -\pi /2,\pi /2];\mathbb{R})$ is the
Sobolev space of even $2\pi $-periodic functions $u(t,x)\ $satisfying
Dirichlet boundary conditions $u(t,\pm \pi /2)=0$. The spectrum of the
elliptic operator $A(\lambda )$ is not explicit, but it can be estimated by
applying the Courant-Fischer-Weyl minmax theorem (Chapter 39 in \cite%
{Rektorys_1977}). These estimates are essential to show that the linear operator 
$L(\lambda _{0})$ has a finite-dimensional kernel, which is non-trivial for
a dense set of values $\lambda _{0}$ in the interval $(0,\lambda _{\ast })$.
We implement a Lyapunov--Schmidt reduction for equation \eqref{op}. The
range equation is solved by the contracting mapping theorem and by proving
that the linear operator is invertible (but not compact) in the range. The
bifurcation equation is solved for a non-resonant value $\lambda _{0}\in
(0,\lambda _{\ast })$ using the Crandall-Rabinowitz theorem \cite%
{Crandall_1971}. The proof of the main theorem is finished by showing that
the number of non-resonant points $\lambda _{0}$ in $(0,\lambda _{\ast })$
is infinite. It is important to mention that the main theorem in \cite{HKiel}
and Remark \ref{R1} imply that the set of bifurcation points $\lambda _{0}$
is dense in $(0,\lambda _{\ast })$, but the bifurcations arising from these
possibly resonant bifurcation points do not satisfy the estimates or
symmetries of our main theorem.

The proof of Theorem \ref{thm:main} requires estimating rigorously the
minimum value of the steady state $u_{\lambda}$ at the critical value $%
\lambda =\lambda_{\ast}$. The proof of a precise estimate for the steady
state $u_{\lambda_*}$ is computer-assisted and is done independently using a
Newton-Kantorovich argument based on the radii polynomial approach (e.g. see 
\cite{MR2338393,LeMi16,notices_jb_jp} and the references therein). To obtain
a rigorous control on $u_{\lambda_{\ast}}$, we use Chebyshev polynomials
series expansions. Since the Chebyshev polynomials are naturally defined on
the interval $[-1,1]$, we rescale the space domain $[-\pi/2,\pi/2]$ to $%
[-1,1]$. Specifically, let $\tilde{U}(t,x;\lambda)$ be a solution of
equation \eqref{1}, then the scaled function $U(t,y;\lambda )\bydef\tilde{U}%
(\alpha t,x;\alpha^{-2}\lambda)$ with $\alpha=\pi/2$ and $x=\alpha y$, is a
solution of the equation%
\begin{equation}
U_{tt}-U_{yy}+\frac{\lambda}{\left( 1+U\right) ^{2}}=0,\qquad y\in
\lbrack-1,1],\qquad U(\pm1)=0.  \label{eq:MEMS}
\end{equation}
We denote the critical value of the scaled equation \eqref{eq:MEMS} by $%
\lambda^{\ast}\bydef\left( \pi/2\right) ^{-2}\lambda_{\ast}.$ The articles 
\cite{Fl3,Ka2} study the model in the domain $[-1/2,1/2]$ and get an exact
implicit formula for the critical value $4\lambda^{\ast}$, which later is
approximated numerically by the value $1.400016469$. Using the radii
polynomial approach, we prove the following result.

\begin{theorem}
\label{thm:saddle-node} The branch of steady states $u_{\lambda}$ of the
MEMS equation \eqref{eq:MEMS} undergoes a saddle-node bifurcation at the
steady state $u_{\lambda^{\ast}}$ satisfying the precise estimates 
\begin{equation*}
\left\vert u_{\lambda^{\ast}}-\bar{u}_{\lambda^{\ast}}\right\vert
_{C^{0}[-1,1]}\leq r=5.7\times10^{-12}, 
\end{equation*}
where $\bar{u}_{\lambda^{\ast}}$ is the numerical approximation portrayed in
Figure \ref{fig:saddle_node_steady_state}, and such that 
\begin{align*}
\lambda^{\ast} & \in0.350004119342744+[-r,r], \\
u_{\lambda^{\ast}}(0) & \in-0.388346718912783+[-r,r].
\end{align*}
\end{theorem}

\begin{figure}[h!]
\centering
\includegraphics[width=0.45\textwidth]{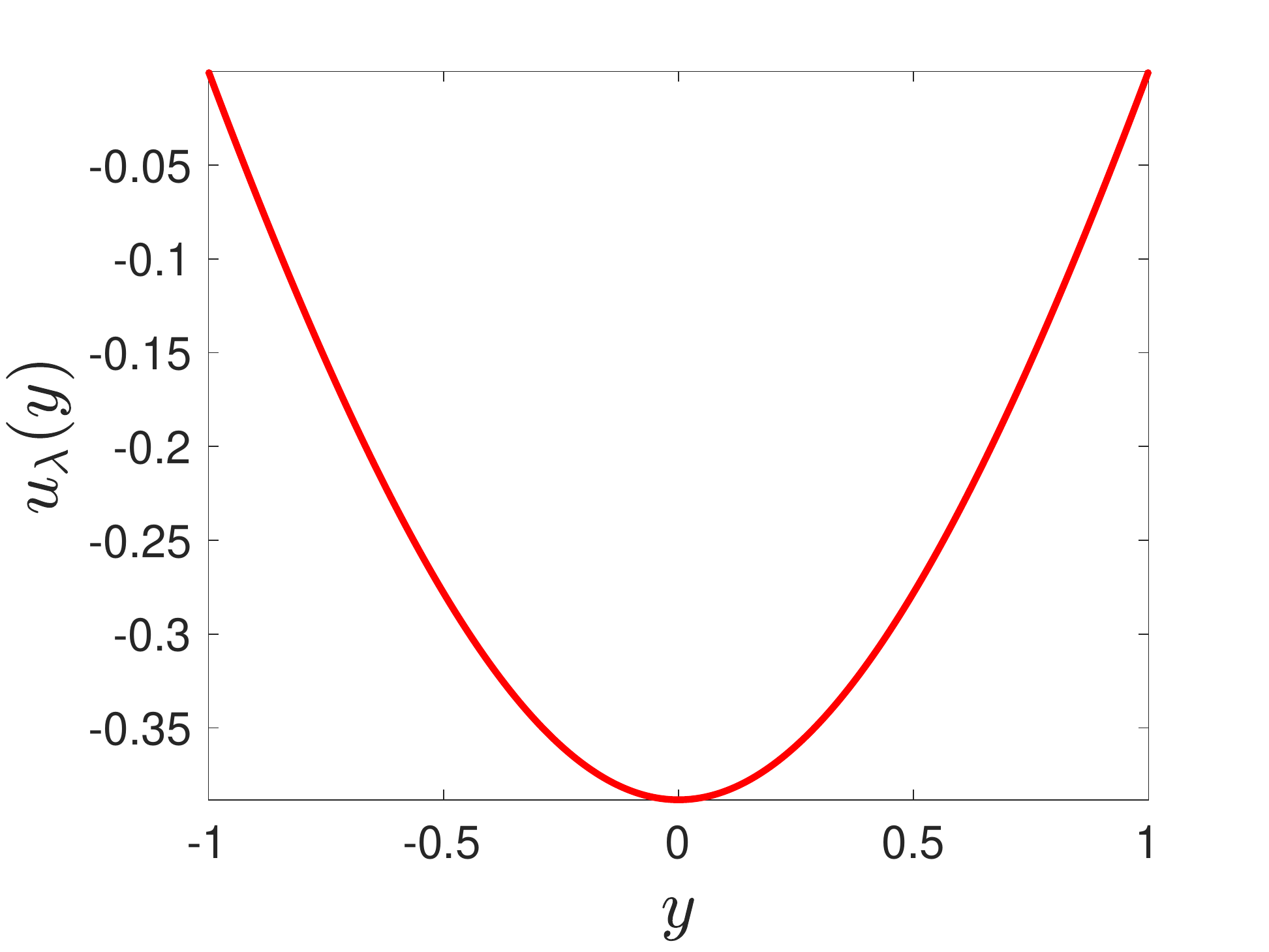} \vspace{-.3cm}
\caption{ The numerical approximation $\bar{u}_{\protect\lambda^{\ast}}(t)$
obtained in Section 3 for the steady state of equation \eqref{eq:MEMS} at
the critical value $\protect\lambda^{\ast}$. }
\label{fig:saddle_node_steady_state}
\end{figure}

\begin{figure}[h!]
\centering
\includegraphics[width=0.4\textwidth]{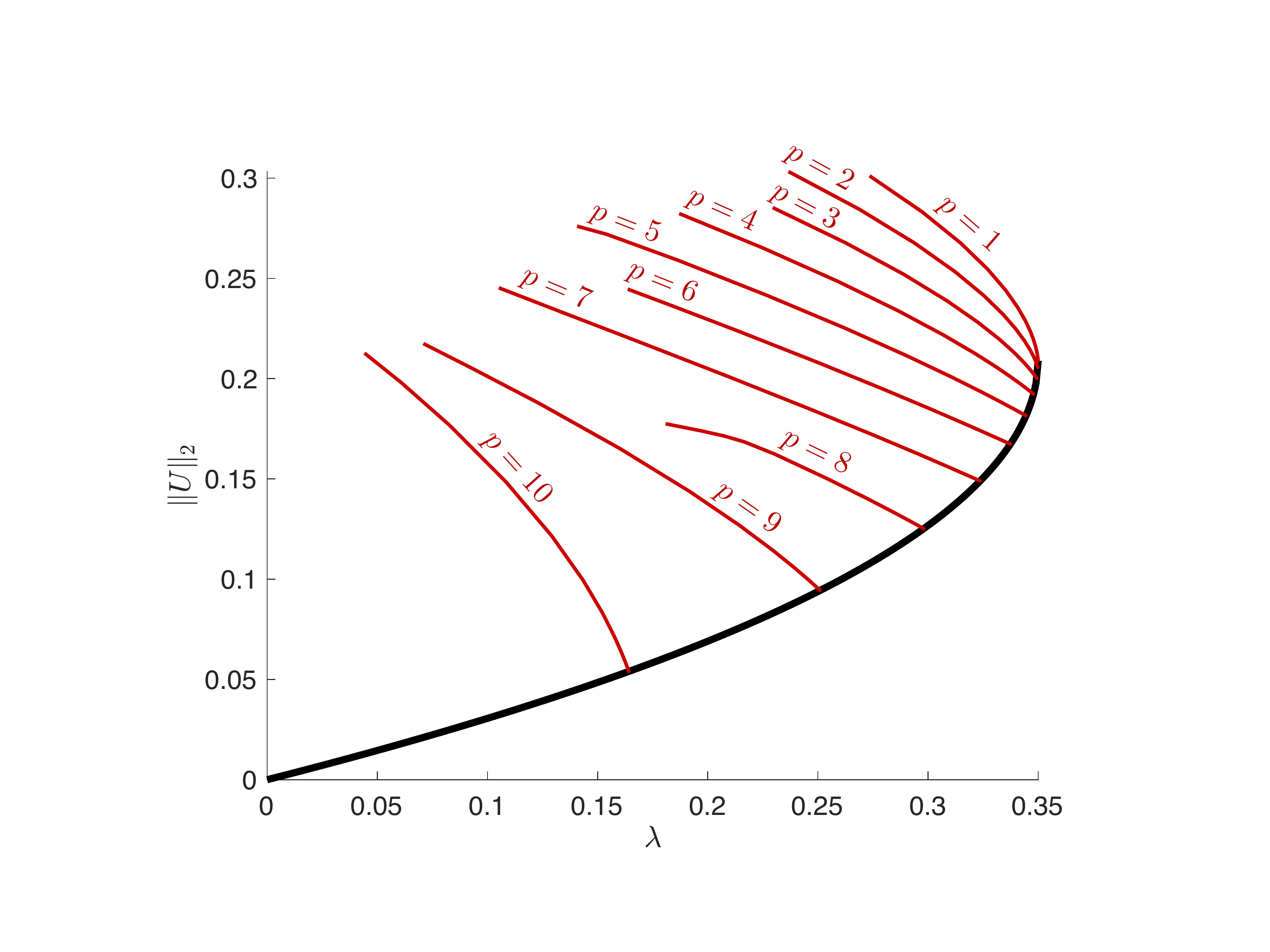}
\caption{Several branches of periodic solutions for $k=1$, $q=11$ and $p
\in\{1,\dots,10\}$.}
\label{fig:branches_po_k_1_q11}
\end{figure}
\begin{figure}[h!]
\centering
\includegraphics[width=0.3\textwidth]{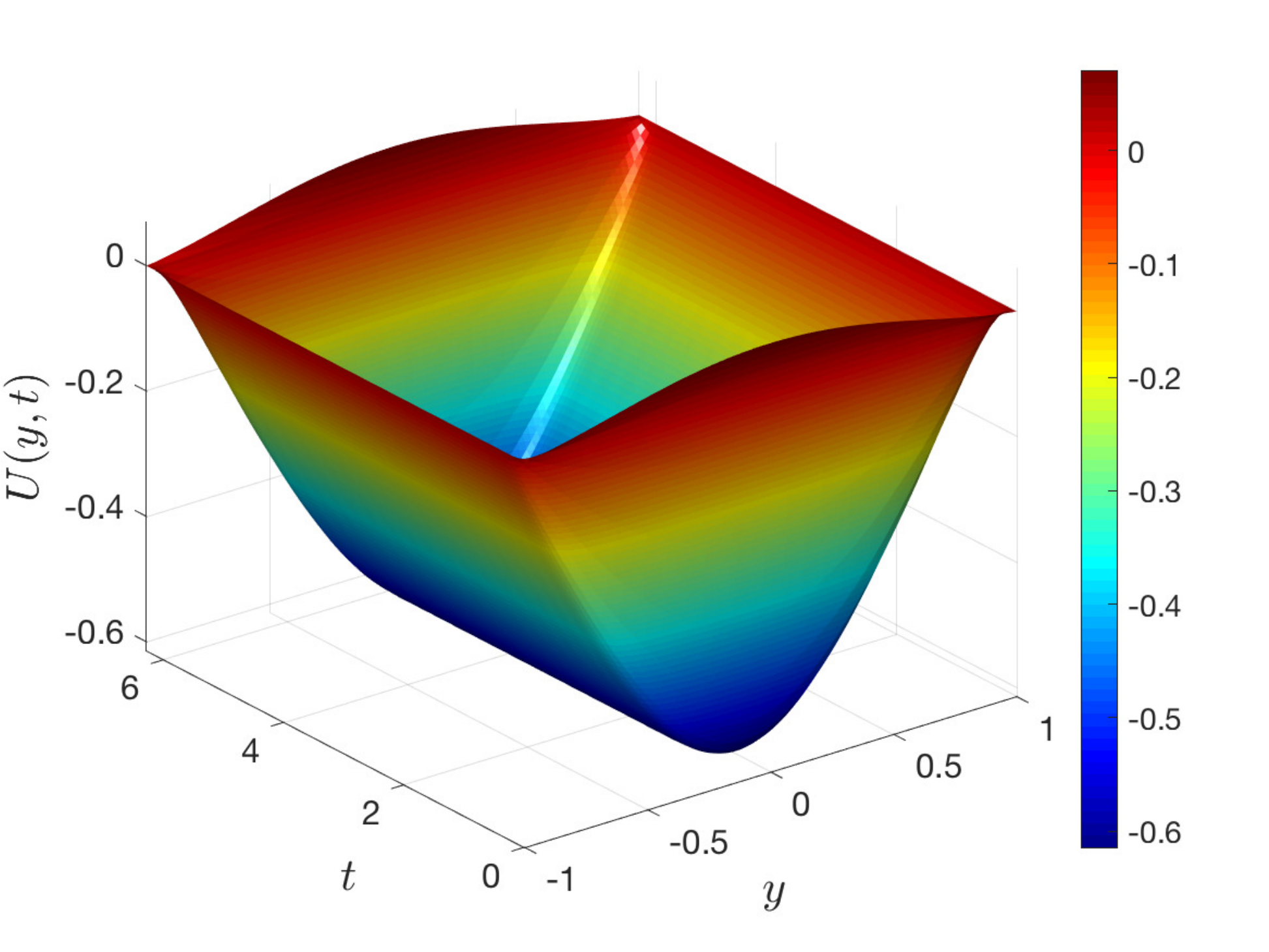} \includegraphics[width=0.3%
\textwidth]{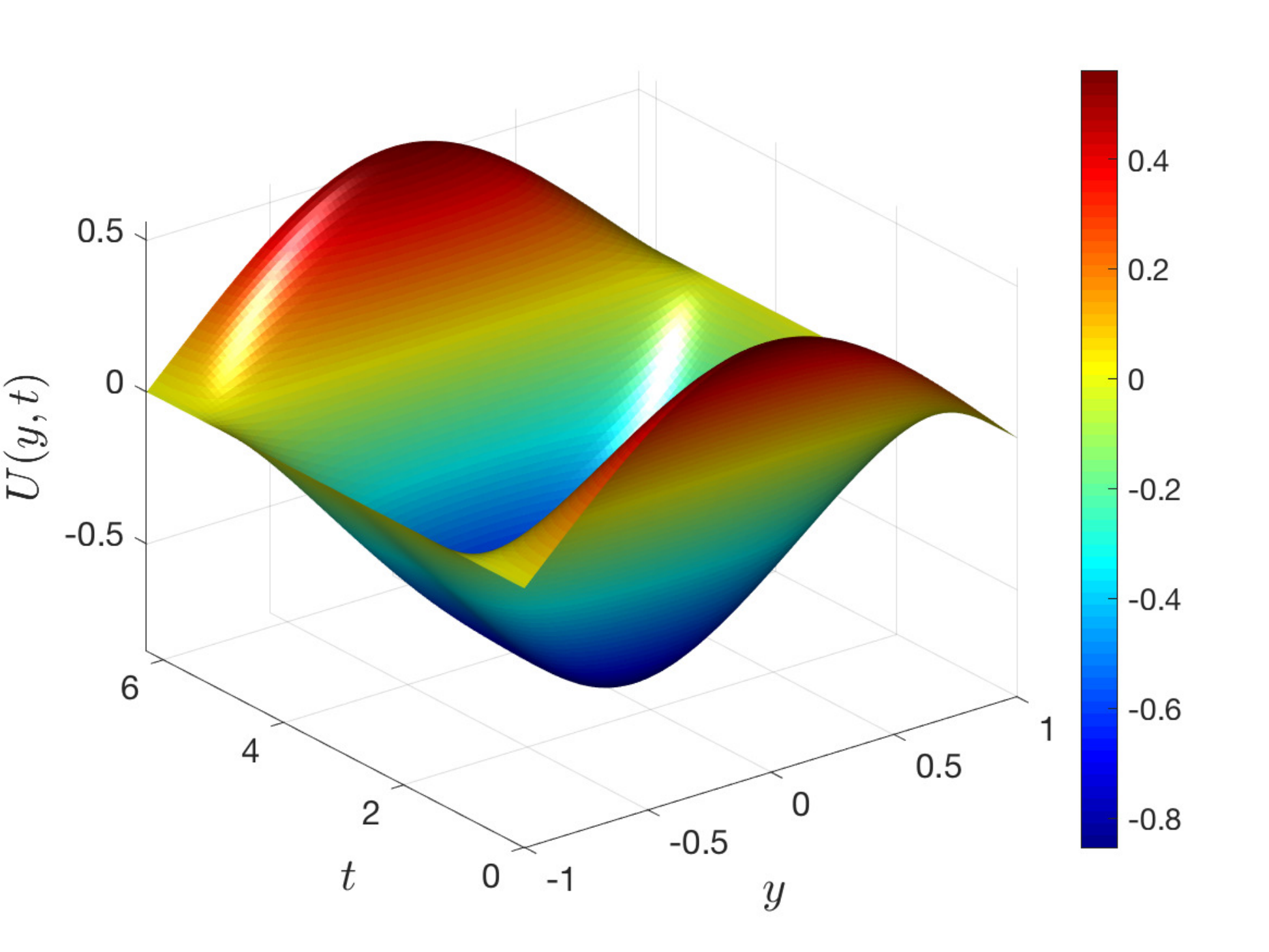} \includegraphics[width=0.3\textwidth]{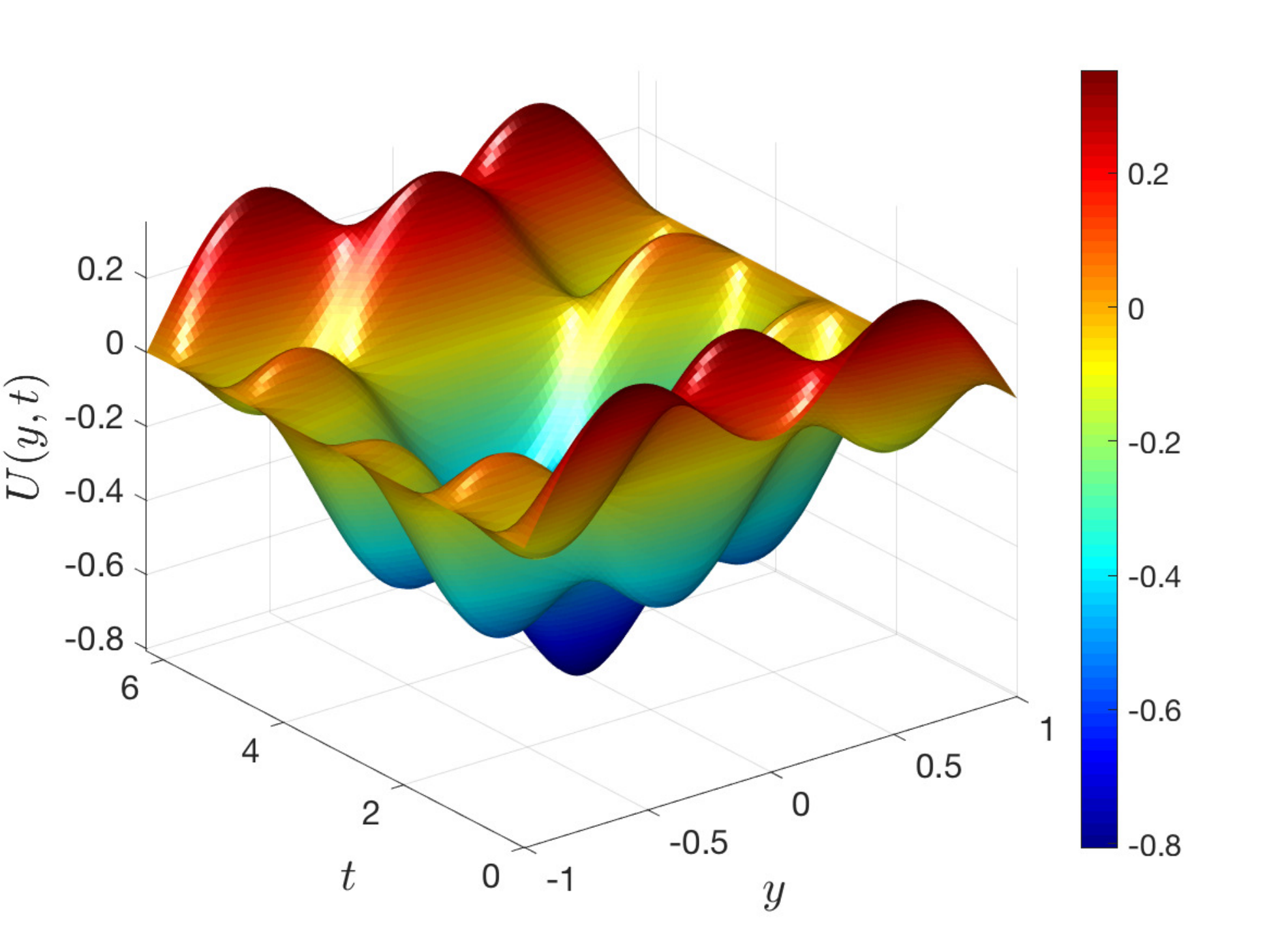}
\caption{Periodic solutions on the branches appearing on Figure~\protect\ref%
{fig:branches_po_k_1_q11} for $p=1$ at $\protect%
\lambda=0.273478006926454$ (\textbf{left}), $p=7$ at $\protect\lambda%
=0.105151105978289$ (\textbf{center}), and $p=8$ at $\protect\lambda%
=0.180819225410784$ (\textbf{right}).}
\label{fig:po_k1_q_11}
\end{figure}
\begin{figure}[h!]
\centering
\includegraphics[width=0.4\textwidth]{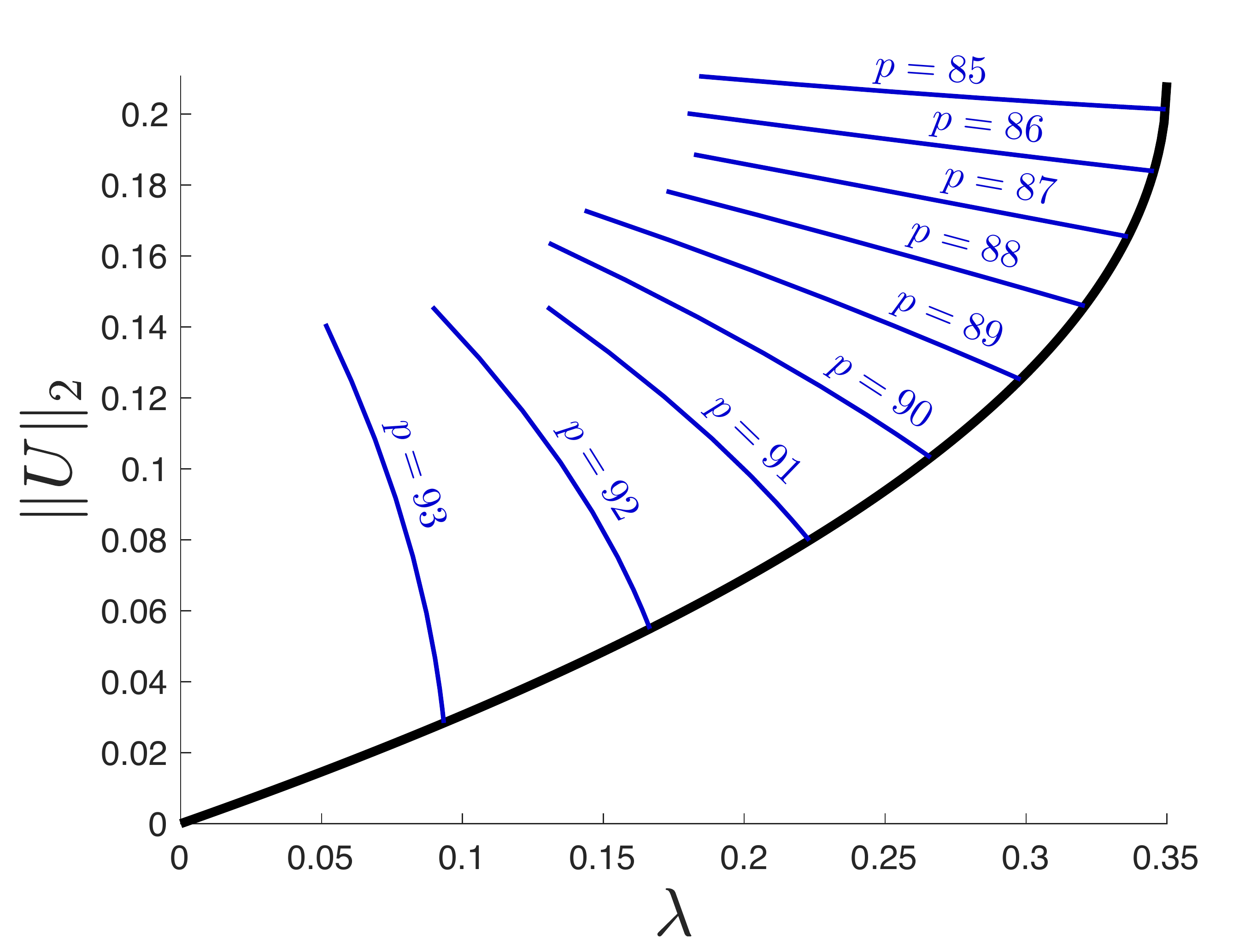}
\caption{Several branches of periodic orbits for $k=2$, $q=47$ and $p \in
\{85,\dots,93\}$.}
\label{fig:branches_po_k_2_q47}
\end{figure}
\begin{figure}[h!]
\centering
\includegraphics[width=0.3\textwidth]{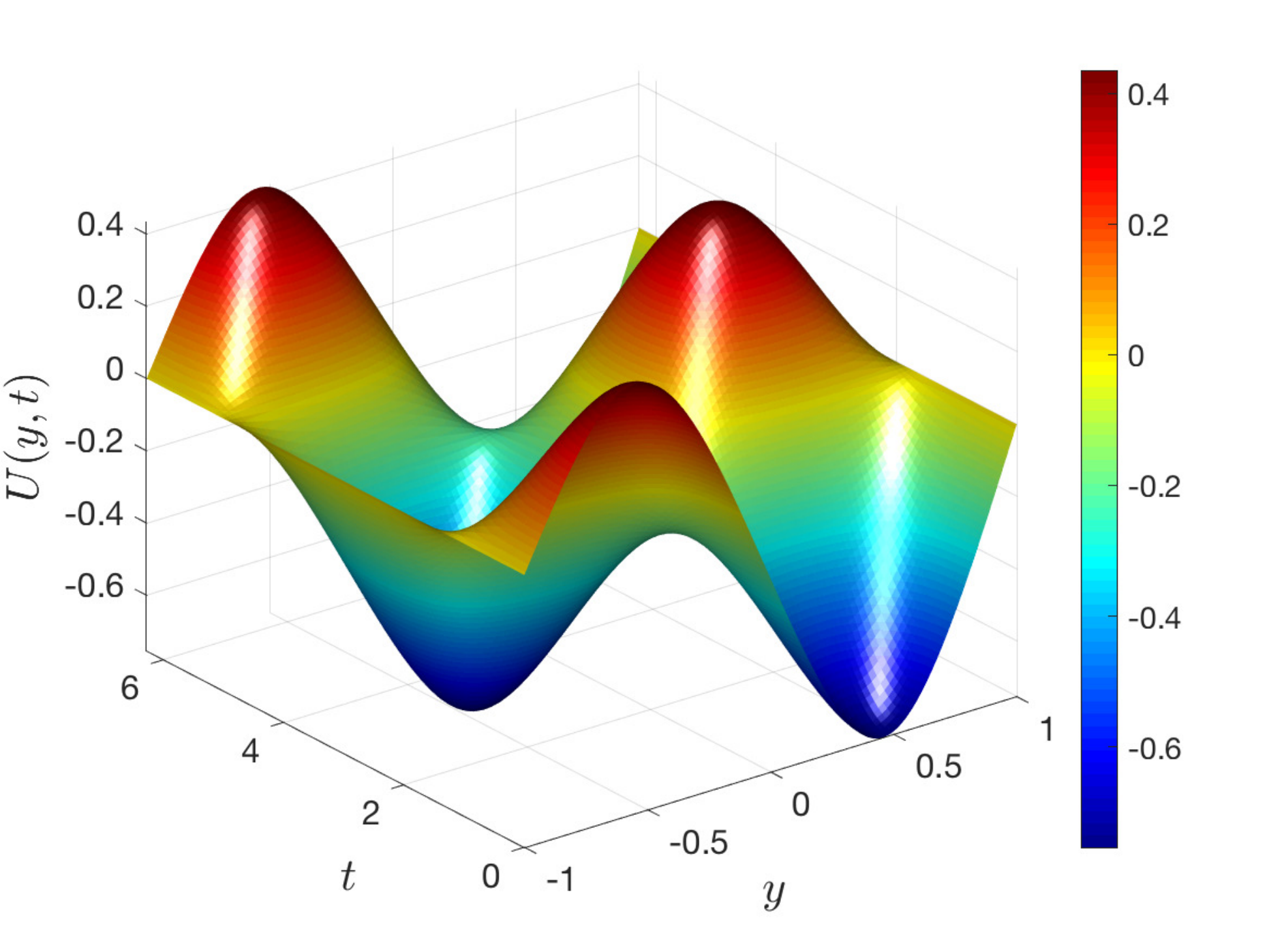} %
\includegraphics[width=0.3\textwidth]{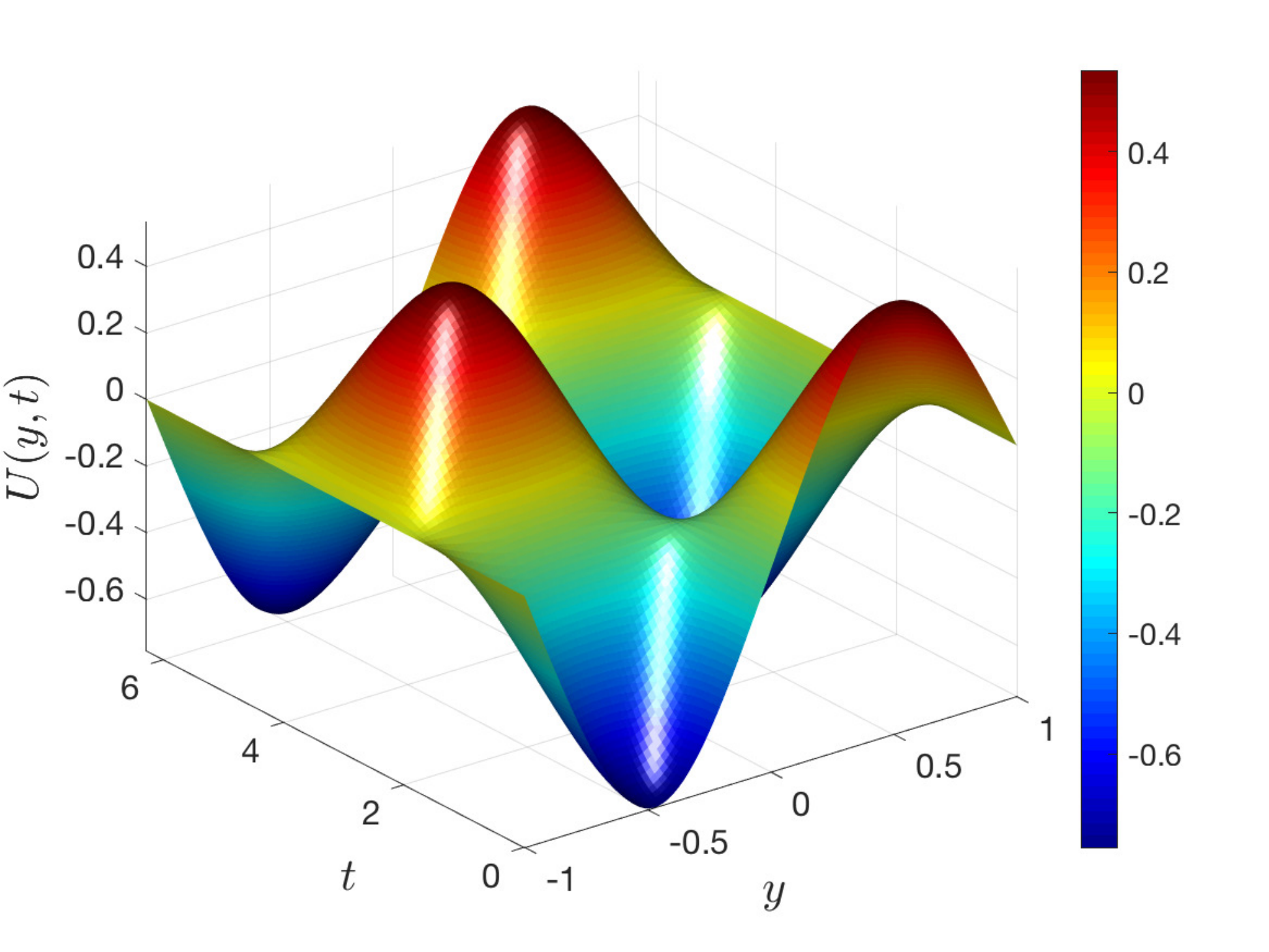}
\caption{The periodic solutions on the branches $p =88$ at $\protect\lambda%
=0.172417160845949$ (\textbf{left}), and $p=90$ at $\protect\lambda%
=0.130674749789634$ (\textbf{right}) of Figure~\protect\ref%
{fig:branches_po_k_2_q47}. Here, $k=2$ and $q=47$.}
\label{fig:po_k2_q_47}
\end{figure}

\pagebreak

The proof of this theorem is performed by solving the boundary value problem %
\eqref{eq:F=0_saddle-node} with polynomial nonlinearities. This setting
simplifies the estimates of the computer-assisted proof by considering the
Banach algebra property of spaces of Chebyshev sequences with geometric
decay (e.g. see \cite{LeMi16,notices_jb_jp}). Note that the computations of
the branch of steady states and the spectrum of the linear elliptic operator
are also performed using Chebyshev series and are set as polynomial boundary
value problems given in (\ref{eq:BVP}) and (\ref{eq:eig_BVP}), respectively.
For the periodic solutions we use the fact that $U(t,x)$ is a $2\pi q/p$%
-periodic solution of equation \eqref{1} if and only if the rescaled
function $U(t,y)$ is a $4q/p$-periodic solution of equation \eqref{eq:MEMS}.
The boundary value problem for periodic solutions of \eqref{eq:MEMS} is
given in \eqref{eq:BVP_PDE}. This setting represents a control problem where
the role of time is taken by the spatial variable $y$ and the control $%
\delta(t)$ is used to determine the initial conditions $U(t,-1)=0$ and $%
U_{y}(t,-1)=\delta(t)$ which guarantee that $U(t,1)=0$. Periodic solutions
are obtained numerically by expanding solutions with Fourier series in time
and with Chebyshev series in space. Examples of periodic solutions computed
numerically for $k=1$ and $q=11$ are portrayed in Figures \ref%
{fig:branches_po_k_1_q11} and \ref{fig:po_k1_q_11}.

It is important to remark that the Newton-Kantorovich argument based on the
radii polynomial approach can be used to validate rigorously the full branch
of steady states $u_{\lambda}$ or the eigenfunctions associated to the
linear elliptic operator $A(\lambda)$ (e.g. using the approach of \cite%
{MR2630003,MR2776917}). On the other hand, the validation of the branches of
periodic solutions requires further investigations due to the lack of
compactness of the inverse of the linear hyperbolic operator $L(\lambda)$.
Indeed, while computer-assisted proofs have been used to validate the
existence of periodic solutions in a nonlinear wave equations in \cite%
{ArioliKoch} and in a nonlinear ill-posed Boussinesq equation (modelling
shallow water waves) in \cite{MR3749257}, in our case the inverse of the
linear hyperbolic operator $L(\lambda)$ lacks the necessary compactness to
apply a similar approach.

The rest of the paper proceeds as follows. In Section~\ref{sec:existence},
we first prove Theorem~\ref{thm:main} by combining a Lyapunov-Schmidt
reduction and the Crandall-Rabinowitz theorem. In Section~\ref%
{sec:numerical_computations}, we compute numerically the steady states, the
spectrum of the elliptic operator and the periodic solutions of equation %
\eqref{eq:MEMS}. Finally, in Section~\ref{sec:cap} we present a
computer-assisted proof of Theorem \ref{thm:saddle-node} that allows
obtaining rigorous control over the steady state $u_{\lambda^{\ast}}$ at the
critical parameter value. 

\section{Existence of an infinite number of branches of periodic solutions}

\label{sec:existence}

\subsection{Properties of the linear elliptic operator}

In this section we analyze the properties of the steady states $u_{\lambda}$
of equation (\ref{1}) and properties of the eigenfunctions of the linear
elliptic operator $A(\lambda)$ defined in (\ref{A}).

\begin{proposition}
The steady state $u_{\lambda}(x)$ is even in $x$ and increasing in the
interval $x\in\lbrack0,\pi/2]$. Moreover, the steady state $u_{\lambda}$ is
decreasing as a function of $\lambda$.
\end{proposition}

\begin{proof}
This follows from the maximum principle, for instance see the results in 
\cite{Fl3} and \cite{Go1}.
\end{proof}

Now we present properties of the eigenvalues and eigenfunctions of the
linear elliptic operator $A(\lambda)$. By Sturm--Liouville theory the
eigenvalues $\mu_{k}(\lambda)$ are simple for all $k$ and $\lambda
\in\lbrack0,\lambda_{\ast}]$. Furthermore, we can order the eigenvalues of $%
A(\lambda)$ by $0\leq\mu_{1}(\lambda)<\mu_{2}(\lambda)<\dots$ such that the
eigenfunction $v_{k}(x;\lambda)$ corresponding to eigenvalue $%
\mu_{k}(\lambda)$ has $k-1$ simple zeros in $(-\pi/2,\pi/2)$.

\begin{remark}
The eigenvalues $\mu_{k}(\lambda)$ for $k=1,2,3$ are computed numerically
for the rescaled problem with $\lambda^{\ast}=\left( \pi/2\right)
^{-2}\lambda_{\ast}$ in Section 3 (see Figure \ref{fig:eigs_continuation}).
Note that, while we do not perform this in the present paper, our numerical
computations of the eigenvalues and eigenfunctions could be validated
rigorously using a Newton-Kantorovich argument based on the radii polynomial
approach similarly to the proof of Theorem \ref{thm:saddle-node} in Section~%
\ref{sec:cap}.
\end{remark}

\begin{proposition}
For $\lambda\in\lbrack0,\lambda_{\ast}]$ we have the following estimate for
the eigenvalues $\mu_{k}(\lambda)$, 
\begin{equation}
2\lambda\leq k^{2}-\mu_{k}(\lambda)\leq9\lambda~.  \label{esei}
\end{equation}
\end{proposition}

\begin{proof}
Let $A,B:H_{0}^{2}\subset L^{2}\rightarrow L^{2}$ be the operators 
\begin{equation*}
A=-\partial_{x}^{2}-B,\qquad B\bydef\frac{2\lambda}{\left( 1+u_{\lambda
}(x)\right) ^{3}}. 
\end{equation*}
Then we have that $m_{\lambda}\left\vert u\right\vert
_{L^{2}}^{2}\leq\left\langle Bu,u\right\rangle \leq M_{\lambda}\left\vert
u\right\vert _{L^{2}}^{2}$, where%
\begin{equation*}
M_{\lambda}\bydef2\lambda\sup_{x\in\lbrack-\pi/2,\pi/2]}\frac{1}{\left(
1+u_{\lambda}(x)\right) ^{3}}\quad\text{and}\quad m_{\lambda}\bydef2\lambda
\inf_{x\in\lbrack-\pi/2,\pi/2]}\frac{1}{\left( 1+u_{\lambda}(x)\right) ^{3}}%
\text{.} 
\end{equation*}
The operator $-\partial_{x}^{2}:H_{0}^{2}\subset L^{2}\rightarrow L^{2}$ has
eigenvalues $k^{2}$ and eigenfunctions $\cos kx$ for $k\in2\mathbb{N}+1$ and 
$\sin kx$ for $k\in2\mathbb{N}^{+}$. Since $-\partial_{x}^{2}=A+B$, the
Courant-Fischer-Weyl theorem implies that%
\begin{equation*}
\mu_{k}(\lambda)+m_{\lambda}\leq k^{2}\leq\mu_{k}(\lambda)+M_{\lambda}\text{.%
} 
\end{equation*}
By the properties of $u_{\lambda}(x)$, we have that $m_{\lambda}=2\lambda$
and $M_{\lambda}=2\lambda\left( 1+u_{\lambda}(0)\right) ^{-3}.$ Since $%
u_{\lambda_{\ast}}(0)$ is estimated in Theorem \ref{thm:saddle-node} with $%
u_{\lambda_{\ast}}(0)>-0.38834671892$, then 
\begin{equation*}
\left( 1+u_{\lambda_{\ast}}(0)\right) ^{-3}<4.5 
\end{equation*}
and 
\begin{equation*}
M_{\lambda}=2\lambda\left( 1+u_{\lambda}(0)\right) ^{-3}\leq2\lambda\left(
1+u_{\lambda_{\ast}}(0)\right) ^{-3}\leq9\lambda~. 
\end{equation*}
\end{proof}

\begin{proposition}
\label{prop:B} The eigenvalues $\mu_{k}(\lambda)$ are monotonically
decreasing for $\lambda\in\lbrack0,\lambda_{\ast}]$.
\end{proposition}

\begin{proof}
Let $\lambda_{1}<\lambda_{2}$. Since $u_{\lambda}$ is decreasing in $\lambda$%
, then 
\begin{equation*}
\frac{\lambda_{1}}{\left( 1+u_{\lambda_{1}}(x)\right) ^{3}}<\frac {%
\lambda_{2}}{\left( 1+u_{\lambda_{2}}(x)\right) ^{3}}\text{.} 
\end{equation*}
Thus the operator%
\begin{equation*}
C\bydef\left( \frac{2\lambda_{2}}{\left( 1+u_{\lambda_{2}}(x)\right) ^{3}}-%
\frac{2\lambda_{1}}{\left( 1+u_{\lambda_{1}}(x)\right) ^{3}}\right)
:H_{0}^{2}\subset L^{2}\rightarrow L^{2}, 
\end{equation*}
is positive definite, that is there exists $c\geq0$ such that $\left\langle
Cu,u\right\rangle _{L^{2}}\geq c\left\langle u,u\right\rangle _{L^{2}}$ for
all $u\in L^{2}$. Applying the Courant-Fischer-Weyl theorem to the operator $%
A(\lambda_{1})=A(\lambda_{2})+C:H_{0}^{2}\rightarrow L^{2}$ we obtain that 
\begin{equation*}
\mu_{k}(\lambda_{1})\geq\mu_{k}(\lambda_{2})+c>\mu_{k}(\lambda_{2})\text{.} 
\end{equation*}
\end{proof}

\begin{proposition}
The eigenfunction $v_{k}(x;\lambda)$ corresponding to the eigenvalue $\mu
_{k}(\lambda)$ of the linear elliptic operator $A(\lambda)$ satisfies 
\begin{equation*}
v_{k}(x;\lambda)=(-1)^{k+1}v_{k}(-x;\lambda)\text{.} 
\end{equation*}
\end{proposition}

\begin{proof}
Since $u_{\lambda}(x)$ is even, then $v_{k}(-x;\lambda)$ is also an
eigenfunction of $A(\lambda)$ for the eigenvalue $\mu_{k}(\lambda)$. Since
the eigenvalues are simple, then the eigenfunctions are unique up to a
scalar multiple, that is $v_{k}(x;\lambda)=\pm v_{k}(-x;\lambda)$. But $%
v_{k}(x;\lambda)$ has $k-1$ simple zeros in $(-\pi/2,\pi/2)$, then $%
v_{k}(x;\lambda)=v_{k}(-x;\lambda)$ if $k$ is odd and $v_{k}(x;%
\lambda)=-v_{k}(-x;\lambda)$ if $k$ is even.
\end{proof}

In particular, we have that 
\begin{equation*}
v_{k}(x;0)=\left\{ 
\begin{aligned}
\cos(kx), \quad  & k\in2\mathbb{N}+1\\
\sin(kx), \quad & k\in2\mathbb{N}^{+}.%
\end{aligned}
\right. 
\end{equation*}

\subsection{Properties of the linear hyperbolic operator}

In this section we analyze the properties of the spectrum for the linear
hyperbolic operator 
\begin{equation}
L(\lambda,\nu)u\bydef\left( p/q\right) ^{2}u_{tt}+A(\lambda)u.  \label{L}
\end{equation}

We define $C_{sym}^{2}$ as the subspace of $2\pi$-periodic even functions $%
u(t,x)\ $satisfying Dirichlet boundary conditions $u(\pm\pi/2)=0$. Thus
functions $u\in C_{sym}^{2}$ have the expansion 
\begin{equation*}
u(t,x)=\sum_{(j,k)\in\mathbb{N}\times\mathbb{N}^{+}}u_{j,k}~\cos
(jt)v_{k}(x;\lambda)\text{,\qquad}u_{j,k}\in\mathbb{R}. 
\end{equation*}
By the estimates (\ref{esei}), the norm%
\begin{equation*}
\left\vert u\right\vert _{s}^{2}\bydef\sum_{(j,k)\in\mathbb{N}\times \mathbb{%
N}^{+}}\left\vert u_{j,k}\right\vert ^{2}\left( j^{2}+k^{2}+1\right) ^{s}%
\text{,} 
\end{equation*}
is equivalent to the standard Sobolev norm. Thus the standard Sobolev space
can be defined by 
\begin{equation*}
H^{s}(S^{1}\times\lbrack-\pi/2,\pi/2];\mathbb{R})\bydef\left\{ u(t,x)\in
L^{2}:\left\vert u\right\vert _{s}<\infty\right\} \text{.} 
\end{equation*}

\begin{definition}
Since the embedding $H^{s}\subset C^{2}$ holds for $s\geq3$, then the
subspace 
\begin{equation}
H_{sym}^{s}\bydef\left\{ u(t,x)\in H^{s}:~~u(t,\pm\pi/2)=0,~~u(t,x)=u(-t,x)\right\}  \label{Hsym}
\end{equation}
is well defined for $s\geq3$.
\end{definition}

The linear hyperbolic map 
\begin{equation*}
L(\lambda,p/q)=(p/q)^{2}\partial_{t}^{2}+A(\lambda):D(L)\subset
H_{sym}^{s}\rightarrow H_{sym}^{s} 
\end{equation*}
is a closed operator, where the subspace $D(L)\subset H_{sym}^{s}$ is closed
under the norm 
\begin{equation*}
\left\vert u\right\vert _{L}^{2}=\left\vert Lu\right\vert
_{s}^{2}+\left\vert u\right\vert _{s}^{2}\text{.} 
\end{equation*}
The linear map $L$ has eigenvalues%
\begin{equation*}
\mu_{j,k}(\lambda,p/q)\bydef-\left( pj/q\right) ^{2}+\mu_{k}(\lambda) 
\end{equation*}
and eigenfunctions $\cos(jt)v_{k}(x)$ for $(j,k)\in\mathbb{N}\times \mathbb{N%
}^{+}$.

\begin{lemma}
\label{lem:lambda} Let $\lambda\in\lbrack\varepsilon,\lambda_{\ast}]$. If $%
pj/q=k$, then $\left\vert \mu_{j,k}(\lambda,p/q)\right\vert \geq2\varepsilon 
$. If $pj/q\neq k$, then%
\begin{equation*}
\left\vert \mu_{j,k}(\lambda,p/q)\right\vert \geq\left( pj/q+k\right)
/q-9\lambda\text{.} 
\end{equation*}
\end{lemma}

\begin{proof}
If $pj/q=k$, then 
\begin{equation*}
\left\vert \mu_{j,k}(\lambda,p/q)\right\vert =k^{2}-\mu_{k}(\lambda
)\geq2\lambda\geq2\varepsilon\text{.} 
\end{equation*}
If $-pj/q+k\neq0$, then $\left\vert -pj/q+k\right\vert \geq1/q$ and%
\begin{equation*}
\left\vert -\left( pj/q\right) ^{2}+k^{2}\right\vert \geq\left\vert
-pj/q+k\right\vert \left\vert pj/q+k\right\vert \geq\left( pj/q+k\right)
/q~. 
\end{equation*}
Thus we have 
\begin{equation*}
\left\vert \mu_{j,k}(\lambda,p/q)\right\vert \geq\left\vert -\left(
pj/q\right) ^{2}+k^{2}\right\vert -\left\vert
\mu_{k}(\lambda)-k^{2}\right\vert \geq\left( pj/q+k\right) /q-9\lambda\text{.%
} 
\end{equation*}
\end{proof}

We define $\mathcal{N}(\lambda,p/q)$ as the set of lattice points where $%
L(\lambda,p/q)$ has a zero eigenvalue, 
\begin{equation}
\mathcal{N}(\lambda,p/q)\bydef\left\{ (j,k):\mu_{j,k}(\lambda,p/q)=0\right\} 
\text{.}  \label{eq:kernel}
\end{equation}

\begin{definition}
Notation $b\lesssim a$ means that there is a positive constant $C$ such that 
$b\leq Ca$.
\end{definition}

\begin{proposition}
If $\lambda\in\lbrack\varepsilon,\lambda_{\ast}-\varepsilon]$, then $%
\left\vert \mu_{j,k}(\lambda,p/q)\right\vert \gtrsim\varepsilon$ for any $%
\left( j,k\right) \notin\mathcal{N}$ and $\mathcal{N}$ is a bounded set with%
\begin{equation}
\mathcal{N}\subset\left\{ (j,k)\in\mathbb{N}^{+}\times \mathbb{N}%
^{+}:pj+qk<9q^{2}\lambda_{\ast}+1\right\} .  \label{bs}
\end{equation}
\end{proposition}

\begin{proof}
Since the first eigenvalue $\mu_{1}(\lambda)$ is positive for $\lambda
\in\lbrack0,\lambda_{\ast})$ and $\mu_{1}(\lambda_{\ast})=0$, then $\mu
_{0,k}(\lambda,p/q)=\mu_{k}(\lambda)\gtrsim\varepsilon$ for $\lambda
<\lambda_{\ast}-\varepsilon$, which implies that $\left( j,k\right) \in%
\mathcal{N}$ only if $j>0$. From Lemma~\ref{lem:lambda}, we have that $%
\left\vert \mu_{j,k}(\lambda,p/q)\right\vert \gtrsim\varepsilon$ for any $%
(j,k)\in\mathbb{N}^{+}\times\mathbb{N}^{+}$ except when $\left(
pj/q+k\right) /q-9\lambda<1/q^{2}$. This inequality is equivalent to 
\begin{equation*}
pj+qk<9\lambda q^{2}+1<9\lambda_{\ast}q^{2}+1. 
\end{equation*}
\end{proof}

\subsection{The Lyapunov-Schmidt reduction}

The $2\pi q/p$-periodic solutions of equation (\ref{1}) of the form $%
U(t,x)=u_{\lambda }(x)+u(pt/q,x)$, where $u(t,x)$ is an even $2\pi $%
-periodic perturbation satisfying Dirichlet boundary conditions $u(t,\pm \pi
/2)=0$, are solutions of the equation%
\begin{equation*}
\left( p/q\right) ^{2}u_{tt}-u_{xx}+\frac{\lambda }{\left( 1+u_{\lambda
}+u\right) ^{2}}-\frac{\lambda }{\left( 1+u_{\lambda }\right) ^{2}}=0.
\end{equation*}%
Thus the equation for the perturbation $u$ from the steady state $u_{\lambda
}$ reads
\begin{equation}
L(\lambda ,p/q)u+g(u)=0\text{,}  \label{EQ}
\end{equation}%
where $L$ is defined in (\ref{L}) and the quadratic
nonlinear operator $g$ is 
\begin{eqnarray}
g(u;\lambda ) &\bydef&\frac{\lambda }{\left( 1+u_{\lambda }+u\right) ^{2}}-\frac{%
\lambda }{\left( 1+u_{\lambda }\right) ^{2}}+\frac{2\lambda }{\left( 1+u_{\lambda }\right) ^{3}}u
\label{g} \\
&=&\lambda \frac{3(1+u_{\lambda })+2u}{\left( 1+u_{\lambda }\right)
^{3}\left( 1+u_{\lambda }+u\right) ^{2}}u^{2}.  \notag
\end{eqnarray}

In this section we make a Lyapunov-Schmidt reduction; namely, we solve the
equation (\ref{EQ}) in the range of the operator $L$ (the range equation)
and we obtain an equivalent equation to (\ref{EQ}) defined in the kernel of $%
L$ (the bifurcation equation).

We start by defining the projection in the kernel of $L(\lambda_{0})$ as 
\begin{equation*}
Qu=\sum_{(j,k)\in\mathcal{N}}u_{j,k}\cos(jt)~v_{k}(x;\lambda
_{0}):H_{sym}^{s}\rightarrow H_{sym}^{s}\text{,} 
\end{equation*}
and the projection $P=I-Q:H_{sym}^{s}\rightarrow H_{sym}^{s}$ in the
complement to the kernel of $L(\lambda_{0})$. We have the following result.

\begin{corollary}
Since $\left\vert \mu_{j,k}(\lambda)\right\vert
^{-1}\lesssim\varepsilon^{-1} $ for $(j,k)\notin\mathcal{N}$, then%
\begin{equation}
\left\vert \left( PLP\right) ^{-1}u\right\vert _{s}\lesssim\varepsilon
^{-1}\left\vert Pu\right\vert _{s}.  \label{Est}
\end{equation}
Thus $\left( PLP\right) ^{-1}:PH_{sym}^{s}\rightarrow PH_{sym}^{s}$ is a
bounded operator. However, the operator $\left( PLP\right) ^{-1}$ is not
compact because the embedding $D(L)\subset H^{s}$ is not necessarily compact.
\end{corollary}

For $s\geq3$, the space $H_{sym}^{s}$ is a Banach algebra. Since $g(u)=%
\mathcal{O}(\left\vert u\right\vert ^{2})$ is analytic in a neighborhood of $%
u=0$, then the nonlinear operator 
\begin{equation*}
g(u)=\mathcal{O}(\left\vert u\right\vert _{s}^{2}):\mathcal{B}_{r}\subset
H_{sym}^{s}\rightarrow H_{sym}^{s} 
\end{equation*}
is continuous and well defined in 
\begin{equation*}
\mathcal{B}_{r}\bydef\{u\in H_{sym}^{s}:\left\vert u\right\vert _{s}<r\}. 
\end{equation*}

Setting 
\begin{equation*}
v=Qu,\qquad w=Pu\text{,} 
\end{equation*}
then $u=v+w$. Thus solutions to equation (\ref{EQ}) are solutions of the
kernel equation%
\begin{equation}
QLQv+Qg(v+w)=0,  \label{EQQ}
\end{equation}
and the range equation%
\begin{equation}
PLPw+Pg(v+w)=0.  \label{EQP}
\end{equation}

The solutions to the range equation \eqref{EQP} in a neighborhood of $%
(v,\lambda)=(0,\lambda_{0})$ are fixed points of the operator 
\begin{equation*}
Kw\bydef\left( PLP\right) ^{-1}g(v+w):P\mathcal{B}_{r}\subset
H_{sym}^{s}\rightarrow H_{sym}^{s}\text{.} 
\end{equation*}
The estimate (\ref{Est}) and the fact that $g(v+w)=\mathcal{O}(\left\vert
w\right\vert _{s}^{2})$ imply that $K$ is a contraction from the domain $P%
\mathcal{B}_{r}$ into itself when we choose $r<<\varepsilon$. By the
contracting mapping theorem there is a unique fixed point $w(v,\lambda)\in
H_{sym}^{s}$ of $K$ for $(v,\lambda)$ in a neighborhood of $(0,\lambda_{0})$%
. Thus there is a unique function $w(v,\lambda)\in H_{sym}^{s}$ that solves
the range equation \eqref{EQP} in a neighborhood of $(0,\lambda_{0})$. We
conclude that the solutions to the equation (\ref{EQ}) are given by the
solutions of the bifurcation equation 
\begin{equation}
QLQv+Qg(v+w(v,\lambda)):Q\mathcal{B}_{r}\subset\ker
L(\lambda_{0})\rightarrow\ker L(\lambda_{0})~.  \label{EQB}
\end{equation}

\subsection{The bifurcation equation}

To solve the bifurcation equation (\ref{EQB}) we need to look for values $%
\lambda_{0}\in(0,\lambda_{\ast})$ such that the linearization $L(\lambda
_{0})$ has a nontrivial kernel, that is $\mu_{j,k}(\lambda_{0},p/q)=0$ for
some lattice point $(j,k)\in\mathcal{N}$.

\begin{proposition}
\label{prop:bifurcation} Define 
\begin{equation*}
B_{k}\bydef\left( \sqrt{ \mu_{k}(\lambda_{\ast})},k\right) \subset\left( 
\sqrt{k^{2}-2\lambda_{\ast}},k\right) ~. 
\end{equation*}
For each rational $p/q\in B_{k}$ there is a \textbf{unique} $\lambda_{0}
\in(0,\lambda_{\ast})$ such that $\mu_{1,k}(\lambda_{0},p/q)=0$.
\end{proposition}

\begin{proof}
Since $\mu_{k}(\lambda;p/q)$ is decreasing and continuous for $\lambda
\in(0,\lambda_{\ast})$, any $\lambda_{0}$ such that $\mu_{1,k}(\lambda
_{0},p/q)=-\left( p/q\right) ^{2}+\mu_{k}(\lambda_{0})=0$ is unique. The
result follows from the fact that the eigenvalue $\mu_{k}(\lambda)$ goes
from the value $\mu_{k}(\lambda_{\ast})\leq k^{2}-2\lambda_{\ast}$ to $\mu
_{k}(0)=k^{2}$ for $\lambda$ in the interval$\ (0,\lambda_{\mathbb{\ast}})$.
\end{proof}

For each $\lambda_{0}\in(0,\lambda_{\ast})$ such that $\mu_{j,k}(\lambda
_{0},p/q)=0$, the set $\mathcal{N}(\lambda_{0},p/q)$ representing the kernel
of $L(\lambda_{0},p/q)$ may contain additional resonant points. If these
resonances exist, they are contained in the bounded set given in (\ref{bs}).

\begin{definition}
We say that $\lambda_{0}$ is a non-resonant value if 
\begin{equation}
\mathcal{N}(\lambda_{0},p/q)=\left\{ (j,k)\in\mathbb{N\times N}%
^{+}:\mu_{j,k}(\lambda_{0},p/q)=0\right\} =\{\left( 1,k\right) \}\text{.}
\end{equation}
If $\lambda_{0}$ is non-resonant, then the kernel has dimension one, that is 
\begin{equation*}
\ker L(\lambda_{0},p/q)=\left\{ b\cos(t)v_{k}(x;\lambda_{0}):b\in \mathbb{R}%
\right\} . 
\end{equation*}
\end{definition}

To prove the existence of a simple bifurcation, we need to choose
non-resonant values $\lambda_{0}$. The following lemma assures the existence
of an infinite set of non-resonant values $\lambda_{0}$.

\begin{proposition}
The set of non-resonant points $\lambda_{0}\in(0,\lambda_{\ast})$ such that 
\begin{equation*}
\mu_{1,k}(\lambda_{0},p/q)=0 
\end{equation*}
for some $p/q\in B_{k}$ is infinite.
\end{proposition}

\begin{proof}
There is a dense set of rationals $p/q\in B_{k}$ such that $%
\mu_{1,k}(\lambda_{0},p/q)=0$ for some $\lambda_{0}\in(0,\lambda_{\ast})$.
Fix one of those points $\lambda_{0}$. By (\ref{bs}), there is at most a
finite number of resonant elements $\left( j_{m},k_{m}\right) \in\mathbb{N}%
^{+}\times\mathbb{N}^{+}$ such that $\mu_{j_{m},k_{m}}(\lambda_{0},p/q)=-%
\left( pj_{m}/q\right) ^{2}+\mu_{k_{m}}\left( \lambda_{0}\right) =0$ for $%
m\in\{0,...,M\}$. Since the eigenvalues $\mu_{k}(\lambda_{0})$ are simple,
then 
\begin{equation*}
\left( pj_{m}/q\right) ^{2}=\mu_{k_{m}}\left( \lambda_{0}\right) \neq
\mu_{k_{0}}\left( \lambda_{0}\right) =\left( pj_{0}/q\right) ^{2}\text{.} 
\end{equation*}
Therefore, the numbers $j_{m}$'s$\ $are different for different numbers $%
m\in\{0,...,M\}$. This implies that there is a unique lattice point denoted
by $(j_{0},k_{0})$ such that $j_{0}$ is maximal, that is $j_{m}<j_{0}$ for $%
m=1,...,M$.

By choosing $p_{0}=pj_{0}$ we have that $\mu_{1,k_{0}}(%
\lambda_{0},p_{0}/q)=0 $ and $\mu_{j,k}(\lambda_{0},p_{0}/q)\neq0$ for all $%
j>1$; otherwise $j_{0}$ would not be maximal. We conclude that any $%
\lambda_{0}$ such that $\ker L(\lambda_{0},p/q)$ is not trivial is
non-resonant for a rational number $p_{0}/q\in B_{k_{0}}$, that is $\mathcal{%
N}(\lambda_{0},p_{0}/q)=\{(1,k_{0})\}$. Moreover, the set of non-resonant
values $\lambda_{0}$ for rational numbers $p_{0}/q\in B_{k_{0}}$ is
infinite, otherwise there has to be at least one point $\lambda_{0}\in(0,%
\lambda_{\ast})$ with an infinite number of resonances, which is a
contradiction to (\ref{bs}).
\end{proof}

\begin{remark}
The choice of the maximal $p_{0}$ is equivalent to the choice of the minimal
period $T=2\pi q/p_{0}$. This argument is similar to the argument used in 
\cite{Ki00}, \cite{Ga19} and \cite{Ga17}.
\end{remark}

\begin{proposition}
\label{pro:main} If $\lambda_{0}$ is a non-resonant value with $%
\mu_{1,k_{0}}(\lambda_{0})=0$, then equation (\ref{EQ}) has a local
bifurcation of zeros from $(u,\lambda_{0})=(0,\lambda_{0})$ such that 
\begin{equation}
u(x,t)=b\left( \cos t\right) v_{k_{0}}(x)+\mathcal{O}_{C_{sym}^{2}}(b^{2}),%
\qquad\lambda=\lambda_{0}+\mathcal{O}(b)~,  \label{u}
\end{equation}
where $b\in\lbrack0,b_{0})$ and $\mathcal{O}_{C_{sym}^{2}}(b^{2})$ is a
function in $C_{sym}^{2}$ of order $b^{2}$.
\end{proposition}

\begin{proof}
The estimate%
\begin{equation*}
v=b\left( \cos t\right) v_{k_{0}}(x)+\mathcal{O}(b^{2})~,\qquad
\lambda=\lambda_{0}+\mathcal{O}(b), 
\end{equation*}
is a consequence of the Crandall-Rabinowitz theorem. Thus the result follows
from the Lyapunov-Schmidt reduction and the fact that $w(v;\lambda )=%
\mathcal{O}_{H_{sym}^{s}}(v^{2})$ with $H_{sym}^{s}\subset C_{sym}^{2}$ for $%
s\geq3$. To apply the Crandall-Rabinowitz theorem, we only need to verify
that $\partial_{\lambda}L(\lambda_{0})\left( \cos t\right) v_{k_{0}}(x)$ is
not in the range of $L(\lambda_{0})$,%
\begin{equation*}
\left\langle \partial_{\lambda}L(\lambda_{0})\left( \cos t\right)
v_{k_{0}}(x),\left( \cos t\right) v_{k_{0}}(x)\right\rangle _{L^{2}}\neq0%
\text{.} 
\end{equation*}
This condition is equivalent to 
\begin{equation*}
\int_{-\pi/2}^{\pi/2}\left( \partial_{\lambda}\frac{\lambda}{\left(
1+u_{\lambda}\right) ^{3}}\right)
_{\lambda=\lambda_{0}}v_{k_{0}}^{2}(x)dx\neq0\text{,} 
\end{equation*}
which follows from the fact that $u_{\lambda}$ is decreasing in $\lambda$.
That is, we have $-\partial_{\lambda}u_{\lambda}\geq0$ and%
\begin{equation*}
\left( \partial_{\lambda}\frac{\lambda}{\left( 1+u_{\lambda}\right) ^{3}}%
\right) _{\lambda=\lambda_{0}}=\frac{1}{\left( 1+u_{\lambda}\right) ^{3}}%
-\partial_{\lambda}u_{\lambda}\frac{\lambda}{\left( 1+u_{\lambda}\right) ^{4}%
}>0. 
\end{equation*}

\end{proof}

\subsection{Proof of Theorem \protect\ref{thm:main}}

The proof of Theorem \ref{thm:main} is a consequence of Proposition \ref%
{pro:main} and the fact that 
\begin{equation*}
U(t,x)=u_{\lambda}(x)+u(pt/q,x). 
\end{equation*}
It only remains to obtain the symmetries of the local bifurcations. This is
a consequence of the following proposition.

\begin{proposition}
The bifurcation (\ref{u}) has the symmetries $u(t,x)=u(t,-x)$ for $k$ odd
and $u(t,x)=u(t+\pi,-x)$ for $k$ even.
\end{proposition}

\begin{proof}
Since $u_{\lambda}(x)$ is even, the equation%
\begin{equation*}
L(\lambda)u+g(u)=0 
\end{equation*}
is equivariant under the action of the group $(\kappa_{1},\kappa_{2})\in%
\mathbb{Z}_{2}\times\mathbb{Z}_{2}$ in $u(t,x)\in H_{sym}^{s}$ given by%
\begin{equation*}
\kappa_{1}u(t,x)=u(t,-x),\qquad\kappa_{2}u(t,x)=u(t+\pi,x)\text{.} 
\end{equation*}
Since 
\begin{equation*}
u(t,x)=\sum_{(j,k)\in\mathbb{N}\times\mathbb{N}^{+}}u_{j,k}~\cos (jt)v_{k}(x)%
\text{,} 
\end{equation*}
where $v_{k}(-x)=(-1)^{k+1}v(x)$ and $\cos j(t+\pi)=\left( -1\right)
^{j}\cos jt$, the actions of $\kappa_{1}$ and $\kappa_{2}$ in the components 
$u_{j,k}\in\mathbb{R}$ are given by $\kappa_{1}u_{j,k}=\left( -1\right)
^{k+1}u_{j,k}$ and $\kappa_{2}u_{j,k}=(-1)^{j}u_{j,k}$. In particular, for $%
j=1$ we have 
\begin{equation*}
\kappa_{1}u_{1,k}=\left( -1\right)
^{k+1}u_{1,k},\qquad\kappa_{2}u_{1,k}=-u_{1,k}~. 
\end{equation*}
Then $u_{1,k}\in\ker L(\lambda_{0})$ is fixed by the action of $\kappa_{1}$
if $k$ is odd and by $\kappa_{1}\kappa_{2}$ if $k$ is even. The result
follows from the fact that $u$ is fixed by the action of $\kappa_{1}$ if it
satisfies that 
\begin{equation*}
u(t,x)=\kappa_{1}u(t,x)=u(t,-x)\text{,} 
\end{equation*}
and by $\kappa_{1}\kappa_{2}$ if%
\begin{equation*}
u(t,x)=\kappa_{1}\kappa_{2}u(t,x)=u(t+\pi,-x). 
\end{equation*}
\end{proof}

\begin{remark}
\label{R1}Since the equation $L(\lambda,p/q)u+g(u;\lambda)=0$ has a gradient
structure and the eigenvalues of $L(\lambda)$ cross zero in the same
direction, because they are decreasing in $\lambda$, then one can use Conley
index to prove that every value $\lambda_{0}$ where the kernel of $%
L(\lambda_{0})$ is not trivial is a bifurcation point \cite{HKiel}.
Therefore, for every $p/q\in B_{k}$ there is a (possibly resonant)
bifurcation value $\lambda_{0}$ such that $\mu_{j,k}(\lambda_{0};p/q)=0$.
Thus the set of bifurcation values $\lambda_{0}$ is dense when considering
all possible values of $p$ and $q$. However, the result using Conley index
does not guarantee that the bifurcation is a continuum satisfying estimates (%
\ref{u}).
\end{remark}

\section{Numerical computation of branches of periodic solutions} \label{sec:numerical_computations}
In this section we compute numerically the steady states, the spectrum of the elliptic operator and the periodic solutions of the scaled equation \eqref{eq:MEMS}. This is done in Sections~\ref{sec:cont_steady-states}, \ref{sec:cont_eigenpairs} and \ref{sec:cont_periodic_solutions}, respectively. For each of these problems, we introduce an infinite dimensional zero-finding problem whose solutions correspond to the wanted objects of interest. Then, a standard predictor-corrector numerical continuation method is applied to finite dimensional projections of each problem. The reason of introducing first the infinite dimensional formulation of the problems is twofold. First, it matches the formulation of Section~\ref{sec:cap} involved in the rigorous computation of the saddle-node bifurcation. Second, it would allow to have the proper formulation for possibly doing rigorous computations of periodic solutions in the future. 

\subsection{Continuation of the branch of steady states} \label{sec:cont_steady-states}

The steady states of the MEMS equation \eqref{eq:MEMS} satisfy the nonlinear boundary value problem
\begin{equation} \label{eq:steady_states_equation}
U_{yy} = \frac{\lambda}{\left(1+U\right)^{2}}, \quad U(-1) = U(1) = 0.
\end{equation}
To compute the solutions of \eqref{eq:steady_states_equation}, we first transform the  equation into a differential equation with polynomial nonlinearities. 
Letting $u_1 \bydef U$, $u_2 \bydef U_y = u_1'$ and $u_3 \bydef \frac{1}{1+u_1}$ yields that
\begin{align*}
u_1' & = u_2 \\
u_2' & = U_{yy} = \frac{\lambda}{\left(1+U \right)^{2}} = \lambda u_3^2 \\
u_3' & = -\frac{1}{(1+u_1)^2} u_1' = -u_2 u_3^2.
\end{align*}
The boundary conditions $u_1(-1)=u_1(1)=0$ are appended. To fix the right condition for $u_3$, we impose that 
$u_3(-1) = \frac{1}{1+u_1(-1)} = 1$. 
The problem of computing a solution $U(y)$ of the nonlinear non-polynomial equation \eqref{eq:steady_states_equation} is then transformed into the polynomial boundary value problem
\begin{equation} \label{eq:BVP}
\begin{pmatrix} u_1' \\ u_2' \\ u_3' \end{pmatrix}
=
\begin{pmatrix} u_2 \\ \lambda u_3^2 \\ -u_2 u_3^2 \end{pmatrix},
\qquad
\begin{pmatrix} u_1(-1) \\ u_2(-1) \\ u_3(-1) \end{pmatrix} = \begin{pmatrix} 0 \\ \delta \\ 1 \end{pmatrix}, \quad u_1(1)=0,
\end{equation}
where $\delta$ is the unknown initial velocity $U'(-1)$ which we will solve for. We expand solutions with Chebyshev series
\[
u_j(y) = (a_j)_0 + 2 \sum_{k\ge 0} (a_j)_n T_n(y), \quad j=1,2,3,
\]
where $T_n:[-1,1] \to \R$ ($n\ge0$) are the Chebyshev polynomials.

Denote by $u=(u_1,u_2,u_3)$ and $\Psi(u)\in \R^3$ the right-hand side of the polynomial differential equation given in \eqref{eq:BVP}.
Denote $a_j=((a_j)_n)_{n\ge 0}$ for $j=1,2,3$, and $a=(a_1,a_2,a_3)$.
For each $j=1,2,3$, the Chebyshev expansion of $\Psi_j(u( \cdot)) : [-1,1] \to \R$ is given by
\[
\Psi_j(u(y)) = (c_j)_0 + 2 \sum_{n\ge 0} (c_j)_n T_n(y), \quad j=1,2,3, \quad y \in [-1,1],
\]
where
\begin{equation} \label{eq:c_j}
\begin{pmatrix} c_1 \\ c_2 \\ c_3 \end{pmatrix}
=
\begin{pmatrix} c_1(a) \\ c_2(\lambda,a) \\ c_3(a) \end{pmatrix}
\bydef
\begin{pmatrix} a_2 \\ \lambda a_3^2 \\ - a_2 a_3^2 \end{pmatrix},
\end{equation}
where $a_3^2 = a_3*a_3$ and $a_2 a_3^2 = a_2*a_3*a_3$ are standard discrete convolutions.

For $j=1,2,3$, let
\begin{equation} \label{eq:IVP_f_operator}
(f^{(\rm eq)}_j(\lambda,\delta,a))_n \bydef \left\{
\begin{aligned}
(a_j)_{0}+2\sum_{\ell=1}^{\infty}(-1)^{\ell}(a_j)_{\ell} - \alpha_j, ~~ & n =0,
\\
 2n (a_j)_n + (c_j(a))_{n+1}-(c_j(a))_{n-1} , ~~& n \ge 1,
\end{aligned}
\right.
\end{equation}
where 
\begin{equation}
\alpha_j \bydef \begin{cases} 0, & j =1 \\ \delta, & j =2 \\ 1, & j =3 \end{cases}.
\end{equation}
Setting $f^{(\rm eq)}_j=((f^{(\rm eq)}_j)_n)_{n \ge 0}$ and $\eta(a_1) = (a_1)_{0}+2\sum_{\ell=1}^{\infty} (a_1)_{\ell}$ (this is the Chebyshev expansion of the extra condition $u_1(1)=0$), the resulting map to solve in the space of Chebyshev coefficients is given by
\begin{equation} \label{eq:f=0_equation}
f^{(\rm eq)}(\lambda,\delta,a) \bydef \begin{pmatrix} \eta(a_1) \\ f^{(\rm eq)}_1(a) \\ f^{(\rm eq)}_2(\lambda,\delta,a) \\ f^{(\rm eq)}_3(a) \end{pmatrix}.
\end{equation}

Define the operators (acting on Chebyshev sequences) by
\begin{equation} \label{eq:tridiagonal_T}
T \bydef
\begin{pmatrix} 
0&0&0&0&0&\cdots\\
-1&0&1&0&\cdots&\ \\
0&-1&0&1&0&\cdots\\
\ &\ddots&\ddots&\ddots&\ddots&\ddots\\
\ &\dots&0&-1&0&1 \\
\ &\ &\dots&\ddots&\ddots&\ddots
\end{pmatrix},
\end{equation}
and
\begin{equation} \label{eq:Lambda}
\Lambda  \bydef
\begin{pmatrix} 
0&0&0&0&0&\cdots\\
0&2&0&0&\cdots&\ \\
0&0&4&0&0&\cdots\\
\ &\ddots&\ddots&\ddots&\ddots&\ddots\\
\ &\dots&0&0&2\ell&0 \\
\ &\ &\dots&\ddots&\ddots&\ddots
\end{pmatrix}.
\end{equation}
%
%
Using the above operators, we may write for the cases $n>0$ 
\[
(f^{(\rm eq)}_j(\lambda,\delta,a))_n  = \left( \Lambda a_j + T c_j(a) \right)_n =  2n (a_j)_n + (c_j(a))_{n+1}-(c_j(a))_{n-1}.
\]
Hence, for $j=1,2,3$,
\[
(f^{(\rm eq)}_j(\lambda,\delta,a))_n \bydef 
\begin{cases}
\displaystyle
(a_j)_{0}+2\sum_{\ell=1}^{\infty}(-1)^{\ell}(a_j)_{\ell} - \alpha_j, ~~ & n =0,\\
\displaystyle
\left( \Lambda a_j + T c_j(a) \right)_n, & n > 0.
\end{cases}
\]

By construction, computing solutions to the nonlinear BVP \eqref{eq:steady_states_equation}
(that is computing equilibria of the MEMS equation) boils down to computing simultaneously $\lambda,\delta,a$ such that $f^{(\rm eq)}(\lambda,\delta,a) =0$ where $f^{(\rm eq)}$ is defined in \eqref{eq:f=0_equation}. Letting $x \bydef (\delta,a)$ we can compute branches of steady states by applying a parameter continuation method (that is a predictor corrector algorithm, see \cite{MR910499}) to a finite dimensional projection of the problem $f^{(\rm eq)}(\lambda,x)=0$, where $\lambda$ is a continuation parameter. 
Figure~\ref{fig:branches_po_k_1_q11} contains the image of the stable branch of steady states (black branch) computed numerically using the presented method.

\subsection{Continuation of eigenfunctions and eigenvalues} \label{sec:cont_eigenpairs}

The eigenfunctions and eigenvalues of the linearized problem can be computed similarly. The  eigenfunctions and eigenvalues are needed in order to find the initial predictor to compute numerically the branches of periodic solutions. The eigenvalue problem associated to \eqref{eq:steady_states_equation} is given by 
\begin{align*} 
- U_{yy} + \frac{\lambda}{\left(1+U\right)^{2}} & = 0, \quad U(-1) = U(1) = 0
\\
-V_{yy} - \frac{2 \lambda}{\left(1+U\right)^{3}} V - \mu V &= 0, \quad V(-1) = V(1) = 0,
\end{align*}
where $(\mu,V)$ is an eigenvalue-eigenvector couple associated to the linearization of the MEMS equation \eqref{eq:MEMS} about the steady state solution $U$.

Letting $u_1 = U$, $u_2 = U_y = u_1'$, $u_3 = \frac{1}{1+u_1}$, $u_4 = V$ and $u_5 = V_y$ yields the system
\begin{equation} \label{eq:eig_BVP}
\begin{pmatrix} u_1' \\ u_2' \\ u_3' \\ u_4' \\ u_5' \end{pmatrix}
=
\begin{pmatrix} u_2 \\ \lambda u_3^2 \\ -u_2 u_3^2 \\ u_5 \\ -2 \lambda u_3^2 u_4 - \mu u_4 \end{pmatrix},
\quad
\begin{pmatrix} u_1(-1) \\ u_2(-1) \\ u_3(-1) \\ u_4(-1) \\ u_5(-1) \end{pmatrix} = \begin{pmatrix} 0 \\ \delta_1 \\ 1 \\ 0 \\ \delta_2 \end{pmatrix}, \quad 
\begin{pmatrix} u_1(1) \\ u_4(1) \end{pmatrix} = \begin{pmatrix} 0 \\ 0 \end{pmatrix},
\end{equation}
where $\delta_1$ and $\delta_2$ are to be (uniquely) determined. Denote $\delta = (\delta_1,\delta_2)$ and 
$u=(u_1,\dots,u_5)$. The unknown variables in the polynomial boundary value problem \eqref{eq:eig_BVP} are 
$(\delta,\mu,u)$. An extra phase condition (that is one which fixes the length of the eigenvector $u_4=V$) will be imposed to isolate the solutions (and therefore allowing the use of Newton's method).

We solve the eigenvalue problem \eqref{eq:eig_BVP} using Chebyshev series expansion, similarly to the BVP \eqref{eq:BVP}. We expand solutions with Chebyshev series
\begin{align*}
u_j(y) &= (a_j)_0 + 2 \sum_{k\ge 0} (a_j)_n T_n(y), \quad j=1,\dots,5.
\end{align*}
Denote by $\Psi(u)\in \R^5$ the right-hand side of the polynomial differential equation given in \eqref{eq:eig_BVP}.
Denote $a_j=((a_j)_n)_{n\ge 0}$ for $j=1,\dots,5$ and $a=(a_1,\dots,a_5)$.
Assume that the Chebyshev expansion of $\Psi(u(y))$ is given by
\[
\Psi(u(y))_j = (c_j)_0 + 2 \sum_{n\ge 0} (c_j)_n T_n(y), \quad j=1,\dots,5 ,
\]
where $c=c(a) = (c_1,\dots,c_5)$ is given component-wise by
\[
\begin{pmatrix} c_1 \\ c_2 \\ c_3 \\ c_4 \\ c_5 \end{pmatrix}
\bydef
\begin{pmatrix} a_2 \\ \lambda a_3^2 \\ - a_2 a_3^2 \\ a_5 \\ -2 \lambda a_3^2 a_4 - \mu a_4 \end{pmatrix},
\]
where $a_3^2 = a_3*a_3$, $a_2 a_3^2 = a_2*a_3*a_3$ and $a_3^2 a_4 = a_3*a_3*a_3*a_4$ are discrete convolutions.

Denote $x \bydef (\delta,\mu,a)$. 
%
For $j=1,\dots,5$, let
\begin{equation} \label{eq:IVP_g_operator}
(g_j)_n(x) \bydef \left\{
\begin{aligned}
(a_j)_{0}+2\sum_{\ell=1}^{\infty}(-1)^{\ell}(a_j)_{\ell} - \alpha_j, ~~ & n =0,
\\
 2n (a_j)_n + (c_j)_{n+1}-(c_j)_{n-1} , ~~& n \ge 1,
\end{aligned}
\right.
\end{equation}
where 
\begin{equation}
\alpha_j \bydef \begin{cases} 0, & j =1 \\ \delta_1, & j =2 \\ 1, & j =3 \\ 0, & j =4 \\ \delta_2, & j = 5 \end{cases}.
\end{equation}
For $j=1,\dots,5$, we set $g_j \bydef ((g_j)_n)_{n \ge 0}$. Let
\begin{align*}
\eta_1(a_1) &\bydef (a_1)_{0}+2\sum_{n=1}^{\infty} (a_1)_{n}
\\
\eta_2(a_4) &\bydef (a_4)_{0}+2\sum_{n=1}^{\infty} (a_4)_{n}
\\
\eta_3(a_4) &\bydef l(a_4)-1,
\end{align*}
where $l$ is linear and acting as a phase condition for the eigenvector $V$ (by fixing its length). Set $\eta \bydef (\eta_1,\eta_2,\eta_3) \in \R^3$. The three extra conditions $\eta(a) = 0 \in \R^3$ are the extra conditions (in Chebyshev) enforcing that $u_1(1)=0$, $u_4(1)=0$ and that the eigenvector $u_4=V$ is locally isolated. The resulting map to solve in the space of  Chebyshev coefficients is given by
\begin{equation} \label{eq:f=0_equation_sn}
f^{(\rm lin)}(x,\lambda) \bydef \begin{pmatrix} \eta(a) \\ g(x,\lambda) \end{pmatrix}.
\end{equation}

We can then apply a standard predictor-corrector method to continue the eigenvalues $\mu_k(\lambda)$ for $k=1,2,3$ and for $\lambda \in [0,\lambda^*)$. Having fixed $k \in \{1,2,3\}$, we begin the continuation at $\lambda=0$ knowing theoretically that at the steady state $U=0$, the eigenvalues are given by $\mu=\mu_k(0) = \left( \frac{\pi k}{2} \right)^2$, $k \ge 0$ with corresponding eigenvectors given by 
\begin{equation} \label{eq:V_at_0}
V(y) = V_k(y) = 
\begin{cases}
\sin \left( \frac{k \pi y}{2} \right), & k \text{ even},
\\
\cos \left( \frac{k \pi y}{2} \right), & k \text{ odd}.
\end{cases}
\end{equation}
Hence, recalling that $u_1 = U$, $u_2 = U_y = u_1'$, $u_3 = \frac{1}{1+u_1}$, $u_4 = V$ and $u_5 = V_y$, we get that at $\lambda=0$, $u_1 \equiv 0$, $u_2 \equiv 0$, $u_3 \equiv 1$, $u_4(y) = V(y)$ as in \eqref{eq:V_at_0} and
\begin{equation} \label{eq:u5_lambda=0}
u_5(y) = V'(y)=V'_k(y) = 
\begin{cases}
\frac{k \pi}{2} \cos \left( \frac{k \pi y}{2} \right), & k \text{ even},
\\
- \frac{k \pi }{2} \sin \left( \frac{k \pi y}{2} \right), & k \text{ odd}.
\end{cases}
\end{equation}

\begin{figure}
\centering
\includegraphics[width=0.4\textwidth]{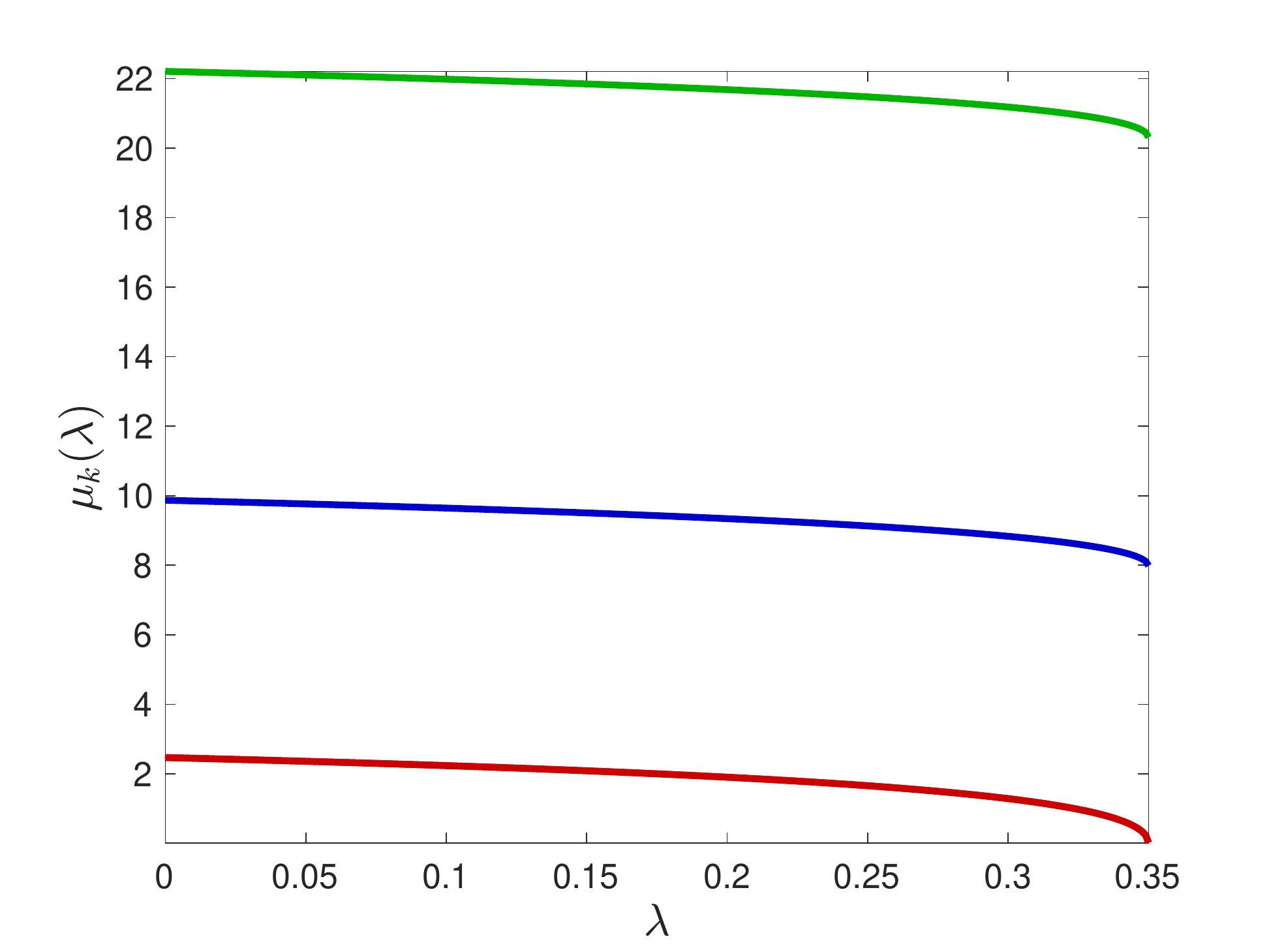}
\vspace{-.3cm}
\caption{Continuation of the eigenvalues $\mu_k(\lambda)$ for $k=1,2,3$, $\lambda \in [0,\lambda^*)$.}
\label{fig:eigs_continuation}
\end{figure}
Using these explicit formulas for $u_j(y)$, $j=1,\dots,5$, we compute the corresponding Chebyshev series expansions to obtain the sequences $\ba_1,\ba_2,\dots,\ba_5$. Note that the sequences $\ba_1=\ba_2=0$ and $(\ba_3)_n = \delta_{n,0}$, where $\delta_{i,j}$ is the Kronecker delta function. The computation of the Chebyshev coefficients $\ba_4$ and $\ba_5$ can be done analytically or using a numerical software. In our case, we use {\tt Chebfun} to compute $\ba_4$ and $\ba_5$. Moreover, we fix $\bar \delta = (\bar \delta_1,\bar \delta_2) = (u_2(-1),u_5(-1))=(0,u_5(-1))$, where $u_5(-1)$ is determined exactly using \eqref{eq:u5_lambda=0}. Letting $\ba=(\ba_1,\dots,\ba_5)$, $\bar \mu = \left( \frac{\pi k}{2} \right)^2$, we have an approximate solution $\bx = (\bar \delta,\bar \mu,\ba)$ which satisfies $f^{(\rm lin)}(\bx,0) \approx 0$. From that approximate solution at $\lambda=0$ at a given $k \in \{1,2,3\}$, we perform a predictor-corrector continuation method on a finite dimensional projection of \eqref{eq:f=0_equation_sn} to obtain a family of solutions of the form $\{ x^{(k)}(\lambda) : \lambda \in [0,\lambda^*) \}$. Denote the second component of $x^{(k)}(\lambda)$ by $\mu_k(\lambda)$, we obtain a branch of eigenvalues parameterized over $\lambda \in [0,\lambda^*)$. See Figure~\ref{fig:eigs_continuation} for a picture of the three branches $k=1,2,3$.

\subsection{Continuation of periodic solutions} \label{sec:cont_periodic_solutions}

The goal now is to compute periodic orbits $U=U(y,t)$ of \eqref{eq:MEMS}. We fix two relatively prime integer $(p,q)$ and fix a priori the frequency to be $\omega \bydef \frac{\pi p}{2q}$. We aim at computing $\frac{2 \pi}{\omega}$-periodic orbits of \eqref{eq:MEMS}.
First, let us transform the problem into a polynomial one. Letting $U_1(y,t) \bydef U(y,t)$, 
$U_2(y,t) \bydef U_y(y,t) = \frac{d}{dy} U_1(y,t)$ and $U_3(y,t) \bydef \frac{1}{1+U(y,t)} = \frac{1}{1+U_1(y,t)}$,
\begin{align*}
\frac{\partial}{\partial y} U_1 & = U_2 \\
\frac{\partial}{\partial y} U_2 & = U_{yy} = \frac{\lambda}{\left(1+U \right)^{2}} + U_{tt} = \lambda U_3^2 + \frac{\partial^2}{\partial t^2} U_1 \\
\frac{\partial}{\partial y} U_3 & = - \frac{1}{(1+U(y,t))^2} U_y(y,t) = -U_2 U_3^2.
\end{align*}
The boundary conditions $U_1(-1,t)=U_1(1,t)=0$ for all $t \in \R$ are appended. To fix the right condition for $U_3$, we impose that 
$U_3(-1,t) = \frac{1}{1+U(-1,t)} = 1$ for all $t \in \R$.

After rescaling time from $[0,\frac{2 \pi}{\omega}]$ to $[0,2\pi]$, the problem of computing a solution $U(y,t)$ of the nonlinear hyperbolic equation \eqref{eq:MEMS} is therefore transformed into finding a $2\pi$-periodic orbit of the polynomial boundary value problem
\begin{equation} \label{eq:BVP_PDE}
\frac{\partial}{\partial y}\begin{pmatrix} U_1 \\ U_2 \\ U_3 \end{pmatrix}
=
\begin{pmatrix} U_2 \\ \lambda U_3^2 + \omega^2 \frac{\partial^2 }{\partial t^2}U_1 \\ -U_2 U_3^2 \end{pmatrix},
\quad
\begin{pmatrix} U_1(-1,t) \\ U_2(-1,t) \\ U_3(-1,t) \end{pmatrix} = \begin{pmatrix} 0 \\ \delta(t) \\ 1 \end{pmatrix}, \quad U_1(1,t)=0, 
\end{equation}
for all $t \in \R$, where $\delta(t)$ is the apriori unknown initial velocity $U_y(-1,t)$ which we will solve for. 
We look for periodic orbits which are even in time. We expand $U_j$ ($j=1,2,3$) and $\delta$ in the form
\begin{align} \label{eq:cheb_fourier_expansion_U}
U_j(y,t) &= \sum_{n \ge 0 \atop k \ge 0} (a_j)_{n,k} m_{n,k} T_{n}(y) \cos(kt)  = 
\sum_{(n,k)\in \Z^2} (a_j)_{n,k} e^{i(n \theta +kt)} \\
\delta(t) &= \sum_{k \ge 0} \delta_{k} m_{0,k} \cos(kt)  = 
\sum_{ k\in \Z} \delta_{k} e^{ikt}
\end{align}
where $\theta \bydef \cos^{-1}(y)$,
\[
m_{n,k} \bydef 
\begin{cases} 
1, & n = k =0 \\ 
2 & n = 0 \text{ and } k>0 \\ 
2 & n > 0 \text{ and } k=0 \\ 
4 & n \ne 0 \text{ and } k \ne 0,
\end{cases} 
\]
and where $(a_j)_{n,k} = (a_1)_{|n|,|k|}$ and $\delta_{n,k} = \delta_{|n|,|k|}$ for $(n,k)\in \Z^2$. Denote $a_1=((a_1)_{n,k})_{n,k\ge 0}$, $a_2=((a_2)_{n,k})_{n,k\ge 0}$, $a_3=((a_3)_{n,k})_{n,k\ge 0}$ and $\delta =(\delta_{k})_{k\ge 0}$.
The unknowns are then given by
\[
x \bydef (\delta,a_1,a_2,a_3).
\]

We integrate the BVP \eqref{eq:BVP_PDE} in $x$ to get the integral formulation 
\begin{equation} \label{eq:integral_PDE}
\begin{pmatrix} U_1(y,t) \\ U_2(y,t) \\ U_3(y,t) \end{pmatrix}
=
\begin{pmatrix} 0 \\ \delta(t) \\ 1 \end{pmatrix}+
\int_{-1}^y
\begin{pmatrix} U_2(\xi,t) \\ \lambda U_3^2(\xi,t) + \omega^2 \frac{\partial^2 }{\partial^2 t}U_1(\xi,t) \\ -U_2(\xi,t) U_3(\xi,t)^2 \end{pmatrix} d \xi, \quad y \in [-1,1]
\end{equation}
supplemented with the boundary condition $U_1(1,t)=0$. 
We denote by $\Psi(U_1,U_2,U_3)$ the right-hand side of the polynomial problem \eqref{eq:BVP_PDE}. That is, 
\[
\Psi
\begin{pmatrix} U_1(y,t) \\ U_2(y,t) \\ U_3(y,t) \end{pmatrix} \bydef
\begin{pmatrix} U_2 \\ \lambda U_3^2 + \omega^2 \frac{\partial^2 }{\partial t^2}U_1 \\ -U_2 U_3^2 \end{pmatrix},
\]
and for $j=1,2,3$, we have the expansion
\[
\Psi_j\begin{pmatrix} U_1(y,t) \\ U_2(y,t) \\ U_3(y,t) \end{pmatrix} =  \sum_{n \ge 0 \atop k \ge 0} (c_j)_{n,k} m_{n,k} T_{n}(y) \cos(kt)  = 
\sum_{(n,k)\in \Z^2} (c_j)_{n,k} e^{i(n \theta +kt)},
\]
where the terms $c_1$, $c_2$ and $c_3$ involve discrete convolution terms. Explicitly, 
\begin{align*}
(c_1)_{n,k} & = (a_2)_{n,k} ,
\\
(c_2)_{n,k} & = \lambda (a_3^2)_{n,k} - \omega^2 k^2 (a_1)_{n,k} ,
\\
(c_3)_{n,k} & = - (a_2 a_3^2)_{n,k},
\end{align*}
where $a_3^2 = a_3*a_3$ and $a_2 a_3^2 = a_2*a_3*a_3$ are standard two-dimensional discrete convolutions;
for instance,
\[
(a*b)_{n,k} = \sum_{{n_1+n_2=n \atop k_1+k_2=k} \atop k_i,n_i \in \Z} a_{n_1,k_1} b_{n_2,k_2}.
\]

For $j=1,2,3$, let
\begin{equation} \label{eq:IVP_g_operator_PO}
(g_j)_{n,k}(x) \bydef \left\{
\begin{aligned}
(a_j)_{0,k}+2\sum_{\ell=1}^{\infty}(-1)^{\ell}(a_j)_{\ell,k} - (\alpha_j)_k, ~~ & n =0,
\\
 2n (a_j)_{n,k} + (c_j)_{n+1,k}-(c_j)_{n-1,k} , ~~& n \ge 1,
\end{aligned}
\right.
\end{equation}
where 
\begin{equation}
(\alpha_j)_k \bydef \begin{cases} 0, & j =1 \\ \delta_k, & j =2 \\ \hat 1_k, & j =3 . \end{cases}
\end{equation}
Setting $g_j=((g_j)_{n,k})_{n,k \ge 0}$ and 
\[
\eta_k(a_1) \bydef (a_1)_{0,k} + 2 \sum_{n \ge 1}^{\infty} (a_1)_{n,k} = 0,
\]
which is the Fourier-Chebyshev expansion of the extra condition $U_1(1,t)=0$ for all $t$. The resulting map to solve in Fourier-Chebyshev coefficients space is given by
\begin{equation} \label{eq:f_po}
f^{(\rm per)}(x,\lambda) \bydef \begin{pmatrix} \eta(a_1) \\ g_1(x) \\ g_2(x,\lambda) \\ g_3(x) \end{pmatrix}.
\end{equation}

Having identified a map whose zeros correspond to periodic orbits, we wish to
compute (once more) branches of solutions (that is of periodic orbits) using a
continuation method. The continuation requires first providing an initial
point. Fix $k\geq1$. For each rational $\lambda_{0}$ such that $\mu
_{k}(\lambda_{0})=\omega^{2}$ with $\omega=\frac{\pi p}{2q}$, there is a local
continuum of $2\pi$-periodic solution bifurcating from the steady solution
$u_{\lambda}(y)$. The initial periodic orbit (that is the predictor) is given by
\[
\hat{u}(t,y)\bydef u_{\lambda_{0}}(y)+b\cos\left(  t\right)  v_{k}\left(
y;\lambda_{0}\right)  ,
\]
for a small $b$. After having transformed this initial point as a sequence of
Fourier-Chebyshev coefficients, we initiate the pseudo-arclength continuation
(e.g. see \cite{MR910499}) on a finite dimensional projection of the map $f^{(\mathrm{per)}}$ defined in
\eqref{eq:f_po}. Using that approach we performed several branch continuation,
which are portrayed in Figures~\ref{fig:branches_po_k_1_q11}
and~\ref{fig:po_k1_q_11} for $k=1$ and $q=11$, and in
Figures ~\ref{fig:branches_po_k_2_q47} and~\ref{fig:po_k2_q_47} for $k=2$ and $q=47$.

\section{Rigorous computation of the saddle-node bifurcation} \label{sec:cap}

In this section, we prove Theorem \ref{thm:saddle-node}. The proof of the theorem is computer-assisted and is based on the successful verification of the contraction mapping theorem of Newton-like operator acting on a ball of radius $r=5.7 \times 10^{-12}$ centered at a numerical approximation of a carefully chosen map defined on a Banach space $X$ of fast decaying Chebyshev coefficients (the {\em saddle-node map} as defined in \eqref{eq:F=0_saddle-node}). 

To define the space $X$ we require first to define the sequence space
\begin{equation} \label{eq:ell_nu_1}
\ell_\nu^1 \bydef \left\{ \alpha = (\alpha_n)_{n \ge 0} : \|\alpha\|_{\nu} = |\alpha_0| + 2 \sum_{n \ge 1} |\alpha_n| \nu^n < 0 \right\},
\end{equation}
for some fixed number $\nu \ge 1$. An important property of $\ell_\nu^1$ is that it is a Banach algebra under discrete convolutions, that is $\| a*b\|_{\nu} \le \|a\|_{\nu} \|b\|_{\nu}$ for all $a,b \in \ell_\nu^1$. This is  useful to perform the necessary estimates to analyze the nonlinear map . 
We also denote  by
\begin{equation} \label{eq:tilde_ell_nu_1}
\tilde \ell_\nu^1 \bydef \left\{ \alpha = (\alpha_n)_{n \ge 0} : \|\alpha\|_{\nu} = |\alpha_0| + 2 \sum_{n \ge 1} |\alpha_n| \frac{\nu^n}{n} < 0 \right\},
\end{equation}
to the corresponding space of Chebyshev coefficients with slightly less decay (regularity) than $\ell_\nu^1$.
Note that for any $\nu \ge 1$ and for a fixed $\lambda$, one can show that the map $f^{(\rm eq)}$ defined in \eqref{eq:f=0_equation} satisfies
\[
f^{(\rm eq)}: \R \times ( \ell_\nu^1 )^3 \to \R \times ( \tilde \ell_\nu^1 )^3.
\]
In order to construct the saddle-node map $F$, we let
\begin{equation} \label{eq:g=0_equation}
g(\lambda,a,\gamma,b) \bydef 
D_{\delta,a}f^{(\rm eq)}(\lambda,\delta,a)({\gamma \atop b}) = 
\begin{pmatrix} \eta(b_1) \\ g_1(b) \\ g_2(\lambda,a,\gamma,b) \\ g_3(a,b) \end{pmatrix}, 
\end{equation}
where 
\[
(g_j)_n \bydef 
\begin{cases}
\displaystyle
(b_j)_{0}+2\sum_{\ell=1}^{\infty}(-1)^{\ell}(b_j)_{\ell}-\tilde \alpha_j,  ~~ & n =0,\\
\displaystyle
\left( \Lambda b_j + T d_j(a,b) \right)_n, & n > 0,
\end{cases}
\]
for $j=1,2,3$, with
\begin{equation} \label{eq:d_j}
\begin{pmatrix} d_1 \\ d_2 \\ d_3 \end{pmatrix}
=
\begin{pmatrix} d_1(b) \\ d_2(\lambda,a,b) \\ d_3(a,b) \end{pmatrix}
\bydef
\begin{pmatrix} b_2 \\ 2 \lambda a_3 b_3 \\ - a_3^2b_2 -2a_2 a_3 b_3 \end{pmatrix},
\end{equation}
and
\begin{equation}
\tilde \alpha_j \bydef \begin{cases} 0, & j =1 \\ \gamma, & j =2 \\ 0, & j =3. \end{cases}
\end{equation}

According to \eqref{eq:f=0_equation}, computing the saddle-node bifurcation point $(\lambda,\delta,a)$ requires solving the augmented system
\begin{equation} \label{eq:F=0_saddle-node}
F(x) \bydef \begin{pmatrix} \ell({\gamma \atop b})-1 \\ f^{(\rm eq)}(\lambda,\delta,a) \\ 
D_{\delta,a}f^{(\rm eq)}(\lambda,\delta,a)({\gamma \atop b}) 
\end{pmatrix}
=
\begin{pmatrix} \ell({\gamma \atop b})-1 \\ f^{(\rm eq)}(\lambda,\delta,a) \\ 
g(\lambda,a,\gamma,b)
\end{pmatrix}
= 0,
\end{equation}
where 
\begin{equation} \label{eq:Banach_space_X}
x \bydef (\lambda,\delta,a,\gamma,b) \in X \bydef \R \times \R \times (\ell_\nu^1)^3 \times \R \times (\ell_\nu^1)^3,
\end{equation}
$b = (b_1,b_2,b_3) \in (\ell_\nu^1)^3$, and $\ell:\R \times (\ell_\nu^1)^3 \to \R$ is a linear functional acting on the eigenvector $({\gamma \atop b})$. We call the map $F$ in \eqref{eq:F=0_saddle-node} the {\em saddle-node map}. By construction, a non-degenerate zero $\tx$ of $F$ yields the existence of a saddle-node point, that is a point such that 
$f^{(\rm eq)}(\lambda,\delta,a)=0$ and such that $D_{\delta,a}f^{(\rm eq)}(\lambda,\delta,a)$ has a one-dimensional kernel.

We endow the space $X$ with the product norm
\begin{equation} \label{eq:Banach_space_X_norm}
\| x\|_X \bydef \max \left\{ |\lambda|,|\delta|,\|a_1\|_{\nu},\|a_2\|_{\nu},\|a_3\|_{\nu},|\gamma|,\|b_1\|_{\nu},\|b_2\|_{\nu},\|b_3\|_{\nu} \right\}.
\end{equation}
Moreover, recalling \eqref{eq:tilde_ell_nu_1}, we define
\begin{equation} \label{eq:Banach_space_Y}
Y \bydef \R \times \R \times (\tilde \ell_\nu^1)^3 \times \R \times (\tilde \ell_\nu^1)^3,
\end{equation}
and one can easily verify that $F:X \to Y$ is well defined.

In this section, we present a computer-assisted approach to solving  the saddle-node map \eqref{eq:F=0_saddle-node} using the tools of rigorous numerics in order to obtain a rigorous control over the value of $\lambda^*$.
This approach will give a proof of Theorem~\ref{thm:saddle-node}. The idea of the computer-assisted proof  is to demonstrate that a certain Newton-like operator is a contraction on a closed ball centered at a numerical approximation $\bx$. To compute $\bx$, we consider a finite dimensional projection of the saddle-node map $F: X  \to Y$.

Given a number $m \in \N$, and given a vector $a = (a_n)_{n \ge 0} \in \ell_\nu^1$, consider the projection
\begin{align*}
\pi^m : \ell_\nu^1& \to \R^{m} \\ 
a &\mapsto \pi^m a \bydef (a_n)_{n=0}^{m-1} \in \R^m.
\end{align*}
Given $N \in \N$, we generalize that projection to get $\pi_N^m:(\ell_\nu^1)^{N} \to \R^{Nm}$ defined by
\[
\pi_N^m(a^{(1)}, \dots, a^{(N)}) \bydef ( \pi^m a^{(1)},\dots,\pi^m a^{(N)}) \in \R^{Nm} ,
\]
and $\Pi^{(m)}: \cX \to \R^{6m+3}$ defined by
\[
\Pi^{(m)} x = \Pi^{(m)} (\lambda,\delta,a,\gamma,b) \bydef \left( \lambda,\delta,\pi_3^m a, \gamma,\pi_3^m b \right) \in \R^{6m+3}.
\]
Often, given $x \in X$, we denote
\[
x^{(m)} \bydef \Pi^{(m)} x \in \R^{6m+3}.
\]

Moreover, we define the natural inclusion $\iota^m : \R^{m} \xhookrightarrow{} \ell_\nu^1$ as follows. For $a = (a_n)_{n=0}^{m-1} \in \R^{m}$, we define $\iota^m a \in \ell_\nu^1$ component-wise by
\[
\left( \iota^m a \right)_{\ell}
= 
\begin{cases}
a_n, & n = 0,\dots,m-1
\\
0, & n \ge m.
\end{cases}
\]
Similarly, let $\iota_N^m:\R^{Nm} \xhookrightarrow{} (\ell_\nu^1)^{N}$ be the natural inclusion defined as follows. Given $a = (a^{(1)},\dots,a^{(N)}) \in (\R^{m})^N \cong \R^{Nm}$, we define
\[
\iota_N^m a \bydef \left( \iota^m a^{(1)}, \dots,\iota^m a^{(N)} \right) \in (\ell_\nu^1)^{N}.
\]
We define the natural inclusion $\inc : \R^{6m+3} \xhookrightarrow{} X$, for $x \in \R^{6m+3}$, by
\[
\inc x = \inc(\lambda,\delta,a,\gamma,b) \bydef \left( \lambda,\delta,\iota_3^m a, \gamma, \iota_3^m b \right) \in X.
\]

Let the {\em finite dimensional projection} $F^{(m)} : \R^{6m+3} \to\R^{6m+3}$ of the saddle-node map \eqref{eq:F=0_saddle-node} be defined, for $x \in \R^{6m+3}$, as 
\begin{equation} \label{eq:F_saddle-node_projection}
F^{(m)}(x) = \Pi^{(m)} F(\inc x).
\end{equation}
%
Assume that a numerical approximation $\bx \in \R^{6m+3}$ of \eqref{eq:F_saddle-node_projection} has been obtained using Newton's method, that is $F^{(m)}(\bx) \approx 0$. We slightly abuse the notation and denote $\bx \in \R^{6m+3}$ and $\inc \bx \in X$ both using $\bx$.

The following result is a Newton-Kantorovich theorem with a smoothing 
approximate inverse. It provides an a-posteriori validation method for 
proving rigorously the existence of a point $\tx$ such that $F(\tx)=0$ 
and $\| \tx - \bx\|_X \le r$ for a small radius $r$. Recalling the norm on 
$X$ given in \eqref{eq:Banach_space_X}, denote by 
\[
B_{r} (y) \bydef \left\{ x \in X : \| x - y \|_X \le r \right\} \subset X
\]
the ball of radius $r$ centered at $y \in X$.

\begin{theorem}[\bf Radii Polynomial Approach] \label{thm:radPolyBanach}
For $\bar x \in X$ and $r > 0$  assume that 
$F: X \to Y$ is Fr\'echet differentiable on the ball $B_r(\bar x)$.
Consider bounded linear operators $A^{\dagger} \in B(X,Y)$ and $A \in B(Y,X)$, 
where $A^{\dagger}$ is an approximation of $D F(\bx)$ and $A$
is an approximate inverse of $D F(\bx)$. 
Observe that 
\begin{equation} \label{eq:AF:X->X}
A F \colon X  \to X.
\end{equation}
Assume that $A$ is injective.
Let $Y_0, Z_0,Z_1,Z_2 \ge 0$ be bounds satisfying
\begin{align}
\label{eq:Y0_radPolyBanach}
\| A F(\bx) \|_X &\le Y_0,
\\
\label{eq:Z0_radPolyBanach}
\| I - A A^{\dagger}\|_{B(X)} &\le Z_0,
\\
\label{eq:Z1_radPolyBanach}
\| A[D F(\bx) - A^{\dagger} ] \|_{B(X)} &\le Z_1,
\\
\label{eq:Z2_radPolyBanach}
\| A[D F(\bx+z) - D F(\bx)]\|_{B(X)} &\le Z_2 r, \quad \forall~ z \in B_r(0).
\end{align}
Define the radii polynomial
\begin{equation} \label{eq:radii_polynomial}
p(r) \bydef Z_2 r^2 + (Z_1 + Z_0 - 1) r + Y_0.
\end{equation}
If there exists $0 < r_0 \leq r$ such that
\begin{equation} \label{eq:p(r0)<0}
p(r_0) < 0,
\end{equation}
then there exists a unique $\tx \in B_{r_0} (\bx)$ such that $F(\tx) = 0$.
\end{theorem}
The proof of the theorem, which is a generalization of the usual Newton-Kantorovich theorem can be found (for example in \cite{MR0231218}).

\subsection{The operators \boldmath$A^\dagger$\unboldmath~and \boldmath$A$\unboldmath} \label{sec:operators_A_dagA}

To apply the radii polynomial approach of Theorem~\ref{thm:radPolyBanach}, we need to define the operator $A^\dagger$ (an approximation of the derivative $DF(\bx)$) and the operator $A$ (an approximation of the inverse of $DF(\bx)$). Consider the finite dimensional projection $F^{(m)}:\R^{6m+3} \to\R^{6m+3}$ given in \eqref{eq:F_saddle-node_projection}, and assume that we computed $\bx \in \R^{6m+3}$ such that $F^{(m)}(\bx) \approx 0$. 

We denote by $DF^{(m)}(\bx) \in M_{6m+3}(\R)$ to the Jacobian matrix of $F^{(m)}$ at $\bx$. For the sake of simplicity, given any $N\in\mathbb{N}$, we denote
the differentiation operator $D$ acting on $u \in (\ell_\nu^1)^N$ as
\begin{equation} \label{eq:differentiation_Chebyshev}
(Du)_{n} \bydef 2 n u_{n} 
=
\begin{pmatrix}
2 n(u_{1})_{n}\\
2 n(u_{2})_{n}\\
\vdots\\
2 n(u_{N})_{n}%
\end{pmatrix}.
\end{equation}
And given $x \in X$, we define 
\begin{equation} \label{eq:A_dagger}
\dagA x = \inc \Pi^{(m)} \dagA x + (I- \inc \Pi^{(m)}) \dagA x,
\end{equation}
where $\Pi^{(m)} \dagA x = DF^{(m)}(\bx) x^{(m)}$ and 
\[
(I- \inc \Pi^{(m)}) \dagA x = 
\begin{pmatrix} 0 \\ 0 \\ 
(I - \iota_3^m \pi_3^m) Da
\\
0
\\
(I - \iota_3^m \pi_3^m) Db
\end{pmatrix}.
\]

Recalling the definition of the Banach space $Y$ in \eqref{eq:Banach_space_Y}, we can verify that the operator $\dagA:X \to Y$ is a bounded linear operator. 
For $m$ large enough, it acts as an approximation of the Fr\'echet derivative $D_xF(\bx)$. Its action on the finite dimensional projection is the Jacobian matrix (the derivative) of $F^{(m)}$ at $\bx$, while its action on the tail only keeps the unbounded terms of the differentiation $D$ defined in \eqref{eq:differentiation_Chebyshev}.

Now we consider a matrix $A^{(m)} \in M_{6m+3}(\R)$ such that $A^{(m)} \approx {DF^{(m)}(\bx)}^{-1}$. In other words, this means that $\| I - A^{(m)} DF^{(m)}(\bx) \| \ll 1$. The computation of $A^{(m)}$ is done using a numerical software ({\tt MATLAB} in our case). We decompose the matrix $A^{(m)}$ block-wise as

\[
A^{(m)}=
\begin{pmatrix}
A^{(m)}_{\lambda,\lambda} & A^{(m)}_{\lambda,\delta} & A^{(m)}_{\lambda,a} & A^{(m)}_{\lambda,\gamma} & A^{(m)}_{\lambda,b}
\\
A^{(m)}_{\delta,\lambda} & A^{(m)}_{\delta,\delta} & A^{(m)}_{\delta,a} & A^{(m)}_{\delta,\gamma} & A^{(m)}_{\delta,b}
\\
A^{(m)}_{a,\lambda} & A^{(m)}_{a,\delta} & A^{(m)}_{a,a} & A^{(m)}_{a,\gamma} & A^{(m)}_{a,b}
\\
A^{(m)}_{\gamma,\lambda} & A^{(m)}_{\gamma,\delta} & A^{(m)}_{\gamma,a} & A^{(m)}_{\gamma,\gamma} & A^{(m)}_{\gamma,b}
\\
A^{(m)}_{b,\lambda} & A^{(m)}_{b,\delta} & A^{(m)}_{b,a} & A^{(m)}_{b,\gamma} & A^{(m)}_{b,b}
\end{pmatrix},
\]
so that it acts is defined on $x^{(m)} = (\lambda,\delta,a^{(m)},\gamma,b^{(m)}) \in \R^{6m+3}$.
Thus we define $A$ as
\begin{equation} \label{eq:A}
A =
\begin{pmatrix}
A_{\lambda,\lambda} & A_{\lambda,\delta} & A_{\lambda,a} & A_{\lambda,\gamma} & A_{\lambda,b}
\\
A_{\delta,\lambda} & A_{\delta,\delta} & A_{\delta,a} & A_{\delta,\gamma} & A_{\delta,b}
\\
A_{a,\lambda} & A_{a,\delta} & A_{a,a} & A_{a,\gamma} & A_{a,b}
\\
A_{\gamma,\lambda} & A_{\gamma,\delta} & A_{\gamma,a} & A_{\gamma,\gamma} & A_{\gamma,b}
\\
A_{b,\lambda} & A_{b,\delta} & A_{b,a} & A_{b,\gamma} & A_{b,b}
\end{pmatrix},
\end{equation}
where the action of each block of $A$ is finite (that is they act on $x^{(m)} = \Pi^{(m)} x$ only) except for the two diagonal blocks $A_{a,a}$ and $A_{b,b}$ which have infinite tails. More explicitly, for each $j=1,2,3$,
\begin{align*}
((A_{a,a} a )_j)_n &= 
\begin{cases}
\bigl( (A_{a,a}^{(m)} \pi_3^m a)_j \bigr)_n & \quad\text{for }   n = 0,\dots m-1, \\
\frac{1}{2 n} (a_j)_n  & \quad\text{for }  n \ge m,
\end{cases}
\\
((A_{b,b} b )_j)_n &= 
\begin{cases}
\bigl( (A_{b,b}^{(m)} \pi_3^m b)_j \bigr)_n & \quad\text{for }   n = 0,\dots m-1, \\
\frac{1}{2 n } (b_j)_n  & \quad\text{for }  n \ge m.
\end{cases}
\end{align*}
Having defined the operators $A$ and $\dagA$, we are ready to define the bounds $Y_0$, $Z_0$, $Z_1$ and $Z_2$ (satisfying \eqref{eq:Y0_radPolyBanach}, \eqref{eq:Z0_radPolyBanach}, \eqref{eq:Z1_radPolyBanach} and \eqref{eq:Z2_radPolyBanach}, respectively), required to built the radii polynomial defined in \eqref{eq:radii_polynomial}.

\subsection{\boldmath$Y_0$\unboldmath~bound} \label{sec:Y0}

Denote by $\bx = (\bar \lambda,\bar \delta,\ba,\bar \gamma,\bb) \in X$ the numerical approximation with 
$\ba=(\ba_1,\ba_2,\ba_3) \in (\ell_\nu^1)^3$ and 
$\bb=(\bb_1,\bb_2,\bb_3) \in (\ell_\nu^1)^3$. 
Recalling the definition of $f$ in \eqref{eq:f=0_equation} (which involves the convolutions in \eqref{eq:c_j}) and the definition of $g$ in \eqref{eq:g=0_equation} (which involves the convolutions in \eqref{eq:d_j}), one has that
\begin{align*}
(I - \iota^{m+1} \pi^{m+1})f_1(\ba) &= 0 \in \ell_\nu^1 ,
\\
(I - \iota^{2m-1} \pi^{2m-1})f_2(\bar \delta,\ba,\bar \lambda) &= 0 \in \ell_\nu^1 ,
\\
(I - \iota^{3m-2} \pi^{3m-2})f_3(\ba) &= 0 \in \ell_\nu^1 ,
\\
(I - \iota^{m+1} \pi^{m+1})g_1(\bb) &= 0 \in \ell_\nu^1 ,
\\
(I - \iota^{2m-1} \pi^{2m-1})g_2(\bar \lambda,\ba,\bar \gamma,\bb) &= 0 \in \ell_\nu^1 ,
\\
(I - \iota^{3m-2} \pi^{3m-2})g_3(\ba,\bb) &= 0 \in \ell_\nu^1 .
\end{align*}
This result follows from the fact that the product of $p$ trigonometric functions of degree $m-1$ is a trigonometric function of degree $p(m-1)$, and the $n$ entry of the Chebyshev map $f_j$ (resp. $g_j$) has entries of the form $(c_j)_{n+1}-(c_j)_{n-1}$ (resp. $(d_j)_{n+1}-(d_j)_{n-1}$). Using that information, one concludes that only finitely many entries of $F(\bx)$ are non-zeros, and therefore the computation of the bound $Y_0$ satisfying 

\begin{equation} \label{eq:Y0_sn}
\| A F(\bx) \|_X \le Y_0
\end{equation}
is a finite computation that can be performed using interval arithmetic ({\tt INTLAB} in our case, see \cite{Ru99a}).

\subsection{Basic functional analytic background}
In this section we present some elementary functional analytic background used to  computing the bounds $Z_0$ and $Z_1$.

For an infinite sequence of real numbers  $a = \{a_n \}_{n \ge 0}$,
and $\nu > 1$, we defined
\[
\| a \|_{\nu} \bydef |a_0| + 2 \sum_{n \ge 1} | a_n | \nu^{n} = \sum_{n \ge 0} | a_n | \omega_n,
\]
where
\[
\omega_n \bydef 
\begin{cases} 
1, & n = 0,
\\
2 \nu^n, & n \ge1.
\end{cases}
\]
The dual norm of $\| \cdot\|_\nu$ is 
\[
\| a \|_{\infty, \nu^{-1}} \bydef \sup_{n \ge 0} \frac{|a_n|}{\omega_n},
\]
and the set
\[
\ell_{\infty, \nu^{-1}} \bydef  \left\{ \{a_n \}_{n \ge 0} \, : \, \| a \|_{\infty, \nu^{-1}}<\infty \right\},
\]
is a Banach space.

We have the following known results (e.g. see \cite{LeMi16}),

\begin{lemma} 
\label{lem:dualityBound_ps}
If $a \in \ell_{\nu}^1$ and 
$c \in \ell_{\nu^{-1}}^\infty$, then
\[
\left| \sum_{n \ge 0}  c_n a_n \right| \leq \|c\|_{\infty,\nu^{-1}} \| a\|_{\nu}.
\]
\end{lemma}

\begin{lemma}\label{lem:fourierCoeffFunctional}
Given $\nu \ge 1$, $k \in \Z$ and $a \in \ell_{\nu}^1$, the function 
$l_{a}^k \colon \ell_{\nu}^1 \to \mathbb{C}$ defined by 
\[
l_a^k(h) \bydef (a*h)_k = \sum_{k_1 + k_2 = k} a_{k_1} h_{k_2},\text{\qquad} h \in \ell_{\nu}^1,
\]
is a bounded linear functional, and 
\begin{equation} \label{eq:l_a^k_norm}
\| l_a^k \| = \sup_{\| h \|_{\nu}\le 1} \left| l_a^k(h) \right| \le
\sup_{j \in \mathbb{Z}} \frac{|a_{k-j}|}{\nu^{|j|}} < \infty.
\end{equation}
\end{lemma}


Fix a truncation mode to be $m$. Given $h \in \ell_\nu^1$, set 
\begin{align*}
h^{(m)} &\bydef (h_0,h_1, \cdots, h_{m-1}, 0, 0, \ldots) \in \ell_\nu^1, \\
h^{(\infty)} &\bydef h - h^{(m)} \in \ell_\nu^1.
\end{align*} 

\begin{corollary}\label{cor:Psi_k}
Let $N \in \N$ and let $\bar \alpha = (\bar \alpha_{0},\bar \alpha_{1}, \cdots, \bar \alpha_{N}, 0, 0, \ldots) \in \ell_\nu^1$.
Suppose that $0 \le k < m$ and define $\hat{l}_{\bar \alpha}^k \in \ell_{\nu^{-1}}^\infty$ by 
\[
\hat{l}_{\bar \alpha}^k (h) \bydef (\bar \alpha * h^{(\infty)})_k  = \sum_{k_1 + k_2 = k } \bar \alpha_{k_1} h^{(\infty)}_{k_2}.
\]
Then, for all $h \in \ell_\nu^1$ such that $\|h\|_{\nu} \le 1$,
\begin{equation} \label{eq:bounds_lkac}
\left| \hat{l}_{\bar \alpha}^k (h) \right| \le 
\Psi_k(\bar \alpha) \bydef
\max_{j=m,\dots,k+m-1} \left( \frac{|\bar \alpha_{k-j}+\bar \alpha_{k+j}|}{2\nu^{j}} \right).
\end{equation}
\end{corollary} 

\begin{proof}
Notice that
\[
\hat{l}_{\bar \alpha}^k (h)  = (\bar \alpha * h^{(\infty)})_k= \sum_{|k_2| \ge m} \bar \alpha_{k-k_2} h_{k_2} 
= \sum_{k_2 \ge m} (\bar \alpha_{k-k_2}+\bar \alpha_{k+k_2}) h_{k_2} 
= \sum_{j \ge 0} c_{j} h_{j},
\]
where
\[
c_j \bydef \begin{cases} \bar \alpha_{k-j} + \bar \alpha_{k+j},& \text{if } j \ge m, \\ 0, & \text{if } 0 \le j \le m. \end{cases}
\]
Since $\bar \alpha_k=0$ for all $|k| > N$, we use \eqref{eq:l_a^k_norm} to obtain
\[
\left| l_{\bar \alpha}^k(h)\right| \le \sup_{j \ge 0} \frac{|c_{j}|}{\omega_j} \le
\max_{j=m,\dots,k+m-1} \frac{|\bar \alpha_{k-j}+\bar \alpha_{k+j}|}{2\nu^{j}}
= \Psi_k(\bar \alpha). 
\]
\end{proof}
\subsection{\boldmath$Z_0$\unboldmath~bound} \label{sec:Z0}

We aim at computing a bound $Z_0$ satisfying \eqref{eq:Z0_radPolyBanach}.
Let $B \bydef I - A A^{\dagger}$, which we denote block-wise by

\[
B =
\begin{pmatrix}
B_{\lambda,\lambda} & B_{\lambda,\delta} & B_{\lambda,a} & B_{\lambda,\gamma} & B_{\lambda,b}
\\
B_{\delta,\lambda} & B_{\delta,\delta} & B_{\delta,a} & B_{\delta,\gamma} & B_{\delta,b}
\\
B_{a,\lambda} & B_{a,\delta} & B_{a,a} & B_{a,\gamma} & B_{a,b}
\\
B_{\gamma,\lambda} & B_{\gamma,\delta} & B_{\gamma,a} & B_{\gamma,\gamma} & B_{\gamma,b}
\\
B_{b,\lambda} & B_{b,\delta} & B_{b,a} & B_{b,\gamma} & B_{b,b}
\end{pmatrix}.
\]
Note that the tails of $B$ vanish by the definition of the diagonal tails of $A$ and $\dagA$. We can compute the bound
\[
{\tiny
Z_0^{(\alpha)} \bydef
\left\{ 
\begin{array}{ll}
\hspace{-.2cm}
\displaystyle
 \sum_{\tilde \alpha \in \{ \lambda,\delta,\gamma \}} 
 \left| B_{\alpha,\tilde \alpha} \right| 
 +  
  \sum_{\tilde \alpha \in \{ a_1,a_2,a_3, \atop \qquad b_1,b_2,b_3 \}} 
  \| B_{\alpha,\tilde \alpha} \|_{\infty, \nu^{-1}}
&
\alpha \in \{ \lambda,\delta,\gamma \}  ,
\\
\hspace{-.2cm}
\displaystyle
\sum_{\tilde \alpha \in \{ \lambda_1,\lambda_2,\lambda_3, \atop \alpha_1,\dots,\alpha_{n-1}\}} 
\hspace{-.2cm}
\| B_{\alpha,\tilde \alpha} \|_\nu
+ \hspace{-.2cm} 
\sum_{\tilde \alpha \in \{ a_1,a_2,a_3, \atop \qquad b_1,b_2,b_3 \}} 
\max_{s=0,\dots,m-1} \frac{1}{\omega_{s}} \sum_{\ell=0,\dots,m-1} \left| \left( B_{\alpha,\tilde \alpha} \right)_{\ell,s} \right| \omega_{\ell} & 
\alpha \in \{ a_1,a_2,a_3,\atop \qquad b_1,b_2,b_3 \}.
\end{array}
\right.
}
\]
letting
\begin{equation} \label{eq:Z0_sn}
Z_0 \bydef
\max \left\{ Z_0^{(\alpha)} : 
\alpha \in \{ \lambda,\delta,\gamma,a_1,a_2,a_3,b_1,b_2,b_3\} \right\},
\end{equation}
by construction we get that 
\[
\| I - A A^{\dagger}\|_{B(X)} \le Z_0.
\]

\subsection{\boldmath$Z_1$\unboldmath~bound} \label{sec:Z1}

For any $h \in B_1(0) \subset X$, let
\[
z \bydef [DF(\bx)-A^\dagger]h \in Y.
\]
Denote 
\begin{equation}  \label{eq:h,z}
\begin{aligned}
h &=(h_\lambda,h_\delta,h_{a_1},h_{a_2},h_{a_3},h_\gamma,h_{b_1},h_{b_2},h_{b_3}),
\\
z &=(z_\lambda,z_\delta,z_{a_1},z_{a_2},z_{a_3},z_\gamma,z_{b_1},z_{b_2},z_{b_3}).
\end{aligned}
\end{equation}
Note that 
\begin{align*}
z_\lambda=0,\text{\qquad}
z_\delta = 2 \sum_{\ell \ge m} (h_{a_1})_\ell,\text{\qquad}
z_\gamma = 2 \sum_{\ell \ge m} (h_{b_1})_\ell.
\end{align*}
Moreover, for $j=1,2,3$, we have
\[
(z_{a_j})_n
=
\begin{cases}
\displaystyle
2 \sum_{\ell \ge m} (-1)^\ell (h_{a_j})_\ell, & n=0
\\
(T \psi_{a_j})_n, & 0<n
\end{cases},
\]
where
\[
(\psi_{a_j})_n \bydef
\begin{cases}
\begin{cases}
0, & 0<n<m \\
(h_{a_2})_n, & n \ge m
\end{cases}, & j =1 
\vspace{.1cm}
\\
\begin{cases}
 \left(2 \bar \lambda \ba_3 h_{a_3}^{(\infty)} \right)_n, & 0<n<m \\
\left( 2 \bar \lambda \ba_3 h_{a_3} + \ba_3^2 h_{\lambda}\right)_n, & n \ge m
\end{cases}, & j =2
\vspace{.1cm}
\\
\begin{cases}
- \left( 2 \ba_2 \ba_3 h_{a_3}^{(\infty)} + \ba_3^2 h_{a_2}^{(\infty)} \right)_n, & 0<n<m \\
- \left( 2 \ba_2 \ba_3 h_{a_3} + \ba_3 ^2 h_{a_2} \right)_n, & n \ge m
\end{cases}, & j =3.
\end{cases}
\]
Similarly, for $j=1,2,3$, we have
\[
(z_{b_j})_n
=
\begin{cases}
\displaystyle
2 \sum_{\ell \ge m} (-1)^\ell (h_{b_j})_\ell, & n=0
\\
(T \psi_{b_j})_n, & 0<n
\end{cases}
\]
where
\[
{\tiny
(\psi_{b_j})_n \bydef
\begin{cases}
\begin{cases}
0, & 0<n<m \\
(h_{b_2})_n, & n \ge m
\end{cases}, & j =1 
\vspace{.1cm}
\\
\begin{cases}
2 \bar \lambda \left( \ba_3 h_{b_3}^{(\infty)} + \bb_3 h_{a_3}^{(\infty)}  
  \right)_n, & 0<n<m \\
\left( 
2 \bar \lambda (\ba_3 h_{b_3} + \bb_3 h_{a_3} ) + 2 \ba_3 \bb_3 h_{\lambda}
\right)_n, & n \ge m
\end{cases}, & j =2
\vspace{.1cm}
\\
\begin{cases}
- \left( 2 \ba_3 \bb_3 h_{a_2}^{(\infty)} + 2 (\ba_3 \bb_3 + \ba_2 \bb_3 )h_{a_3}^{(\infty)} + (\ba_3^2 + 2 \ba_2 \ba_3 ) h_{b_3}^{(\infty)} \right)_n, & 0<n<m \\
- \left( 2 \ba_3 \bb_3 h_{a_2} + 2 (\ba_3 \bb_2 + \ba_2 \bb_3 )h_{a_3} + \ba_3^2h_{b_2} + 2 \ba_2 \ba_3 h_{b_3} \right)_n, & n \ge m
\end{cases}, & j =3.
\end{cases}
}
\]

Using Corollary~\ref{cor:Psi_k}, for each $n = 0,\dots,m$, we can easily compute upper bounds $\hat \psi_{a_j},\hat \psi_{b_j} \ge 0$ such that 
\[
| (\psi_{a_j})_n | \le (\hat \psi_{a_j})_n
\quad \text{and} \quad 
| (\psi_{b_j})_n | \le (\hat \psi_{b_j})_n, \qquad \text{for all } n = 0,\dots,m.
\]
More explicitly, we set $(\hat \psi_1^a)_n = (\hat \psi_1^b)_n =0$ for $n=0,\dots,m$, and
\begin{align*}
(\hat \psi_{a_2})_n & =  2 |\bar \lambda| \Psi_n(\ba_3) ,
\\
(\hat \psi_{a_3})_n & = 2\Psi_n( \ba_2 \ba_3)+ \Psi_n(\ba_3^2) ,
\\
(\hat \psi_{b_2})_n & = 2 |\bar \lambda| \left( \Psi_n(\ba_3) + \Psi_n(\bb_3) \right)  ,
\\
(\hat \psi_{b_3})_n & = 2\Psi_n(\ba_3 \bb_3) + 2 \Psi_n(\ba_3 \bb_2 + \ba_2 \bb_3 ) + \Psi_n(\ba_3^2) + 2 \Psi_n(\ba_2 \ba_3 ).
\end{align*}
Using these bounds, for $n=0,\dots,m-1$, we compute $(\hat z_{a_j})_n$ and $(\hat z_{b_j})_n$ such that 
$|(z_{a_j})_n| \le (\hat z_{a_j})_n$ and $|(z_{b_j})_n| \le (\hat z_{b_j})_n$. More explicitly, for each $ j=1,2,3$, set
\[
(\hat z_{a_j})_0 = (\hat z_{b_j})_0 = \frac{1}{\nu},
\]
and
\[
(\hat z_{a_j})_n \bydef (|T|\hat \psi_{a_j})_n
\qquad 
(\hat z_{b_j})_n \bydef (|T|\hat \psi_{b_j})_n,
\]
for $n=1,\dots,m-1$, where $|T|$ represents the operator with entries given by the component-wise absolute values of the entries of $T$.
Moreover, set 
\begin{align*}
\hat z_\lambda \bydef 0,\text{\qquad}
\hat z_\delta = \frac{1}{\nu},\text{\qquad}
\hat z_\gamma =  \frac{1}{\nu}.
\end{align*}
Recall that $z \bydef [DF(\bx)-A^\dagger]h$. Denote $w = A z$ and 
\[
w =(w_\lambda,w_\delta,w_{a_1},w_{a_2},w_{a_3},w_\gamma,w_{b_1},w_{b_2},w_{b_3}).
\]
Thus, for each $\alpha \in \{a_1,a_2,a_3,b_1,b_2,b_3\}$, we have the estimate
\begin{align*}
\| w_{\alpha}\|_\nu &= \sum_{n=0}^{m-1} \left| \left( \left( A^{(m)} z^{(m)} \right)_\alpha \right)_n \right| \omega_n 
+ \sum_{n \ge m} \frac{1}{2n} |(w_{\alpha})_n| \omega_n
\\
& \le  \sum_{n=0}^{m-1} \left( \left( |A^{(m)}| \hat z^{(m)} \right)_\alpha \right)_n  \omega_n 
+ \frac{1}{2m} \sum_{n \ge m}  |(T \psi_{\alpha})_n| \omega_n
\\
& \le 
z^{(1)}_\alpha \bydef \sum_{n=0}^{m-1} \left( \left( |A^{(m)}| \hat z^{(m)} \right)_\alpha \right)_n  \omega_n 
+ \frac{1}{2m} \| T\|_{B(\ell_\nu^1)} \hat \psi_{\alpha}^{(\infty)},
\end{align*}
where

\[
\hat \psi_{\alpha}^{(\infty)} \bydef
\begin{cases}
2 \bar \lambda \| \ba_3 \|_\nu + \| \ba_3^2\|_\nu, & \alpha = a_2
\\
2 \|\ba_2 \ba_3\|_\nu  +\| \ba_3 ^2 \|_\nu , & \alpha = a_3
\\
2 \bar \lambda (\|\ba_3 \|_\nu+ \|\bb_3 \|_\nu ) + 2 \| \ba_3 \bb_3 \|_\nu , & \alpha = b_2
\\
2 \|\ba_3 \bb_3 \|_\nu + 2 \| \ba_3 \bb_2 + \ba_2 \bb_3 \|_\nu + \| \ba_3^2\|_\nu + 2\| \ba_2 \ba_3 \|_\nu, & \alpha = b_3.
\end{cases}
\]
Therefore, we set 
\begin{equation} \label{eq:Z1_sn}
Z_1 = \max\left\{ z^{(1)}_{\lambda},z^{(1)}_{\delta},z^{(1)}_{a_1},z^{(1)}_{a_2},z^{(1)}_{a_3},z^{(1)}_{\gamma},z^{(1)}_{b_1},z^{(1)}_{b_2},z^{(1)}_{b_3} \right\},
\end{equation}
where 
\begin{align*}
z^{(1)}_{\lambda} \bydef \left( |A^{(m)}| \hat z^{(m)} \right)_\lambda,\text{\qquad}
z^{(1)}_{\delta} \bydef \left( |A^{(m)}| \hat z^{(m)} \right)_\delta,\text{\qquad}
z^{(1)}_{\gamma} \bydef \left( |A^{(m)}| \hat z^{(m)} \right)_\gamma .
\end{align*}
\subsection{\boldmath$Z_2$\unboldmath~bound} \label{sec:Z2}

We look for a bound $Z_2$ in \eqref{eq:Z2_radPolyBanach} such that
\[
\| A[DF(\bx+z) - D F(\bx)]\|_{B(X)} \le Z_2 r, \quad \forall~ z \in B_r(0).
\]
Let $z \in B_r(0)$ and $h \in B_1(0)$ (which we denote component-wise as in \eqref{eq:h,z}), and denote
\[
w = w(z,h) \bydef \left( DF(\bx+z) - D F(\bx) \right)h.
\]
Note that $w_\lambda=w_\delta=w_\gamma=0$, $w_{a_1}=w_{b_1}=0 \in \ell_\nu^1$. Moreover, for each $\alpha \in \{a_2,a_3,b_2,b_3\}$, we have $w_{\alpha} = T \hat w_{\alpha}$ with
\begin{align*}
\hat w_{a_2} & = 2 \ba_3 h_{\lambda} z_{a_3} + 2 h_{a_3} \bar \lambda z_{a_3} + h_{\lambda} z_{a_3}^2 + 2 \ba_3 h_{a_3} z_{\lambda} + 2 h_{a_3} z_{a_3} z_{\lambda} ,
\\
\hat w_{a_3} & = -2 \ba_3 h_{a_3} z_{a_2} - 2 \ba_3 h_{a_2} z_{a_3} - 2 \ba_2 h_{a_3} z_{a_3} - 2 h_{a_3} z_{a_2} z_{a_3} - 
 h_{a_2} z_{a_3}^2 ,
\\
\hat w_{b_2} & = 2 \bb_3 h_{\lambda} z_{a_3} + 2 h_{b_3} \bar \lambda z_{a_3} + 2 \ba_3 h_{\lambda} z_{b_3} + 2 h_{a_3} \bar \lambda z_{b_3}  \\ 
& \quad +  2 h_{\lambda} z_{a_3} z_{b_3} + 2 \bb_3 h_{a_3} z_{\lambda} + 2 \ba_3 h_{b_3} z_{\lambda} + 2 h_{b_3} z_{a_3} z_{\lambda} + 
 2 h_{a_3} z_{b_3} z_{\lambda} ,
\\
\hat w_{b_3} & = 
-2 \bb_3 h_{a_3} z_{a_2} - 2 \ba_3 h_{b_3} z_{a_2} - 2 \bb_3 h_{a_2} z_{a_3} - 2 \bb_2 h_{a_3} z_{a_3} - 
 2 \ba_3 h_{b_2} z_{a_3} 
 \\
 & \quad  - 2 \ba_2 h_{b_3} z_{a_3} - 2 h_{b_3} z_{a_2} z_{a_3} 
 - h_{b_2} z_{a_3}^2  - 2 \ba_3 h_{a_3} z_{b_2} - 2 h_{a_3} z_{a_3} z_{b_2} 
  \\
 & \quad 
 - 2 \ba_3 h_{a_2} z_{b_3} - 2 \ba_2 h_{a_3} z_{b_3} - 
 2 h_{a_3} z_{a_2} z_{b_3} - 2 h_{a_2} z_{a_3} z_{b_3}.
\end{align*}
Thus
\begin{align*}
\| \hat w_{a_2}\|_\nu & \le \left( 4 \|\ba_3\|_\nu  + 2 |\bar \lambda|  + 3  r \right) r ,
\\
\|\hat w_{a_3}\|_\nu & \le  \left( 4 \|\ba_3\|_\nu  + 2 \| \ba_2 \|_\nu + 3 r  \right) r ,
\\
\|\hat w_{b_2}\|_\nu & \le 
\left(
4 \| \bb_3\|_\nu + 4 \|\ba_3\|_\nu + 4 |\bar \lambda|  + 6 r \right) r ,
\\
\|\hat w_{b_3}\|_\nu & \le \left( 4 \|\ba_2\|_\nu + 2 \|\bb_2\|_\nu  + 4 \|\bb_3\|_\nu  + 8 \|\ba_3\|_\nu + 9 r  \right) r.
\end{align*}
Note that $\| T \|_{B(\ell_\nu^1)} \le 2 \nu$, because 
\begin{align*}
\| T h \|_\nu & = 2 \sum_{j \ge 1} |-h_{j-1}+h_{j+1}| \nu^j 
\le  \nu  \left( 2 \sum_{j \ge 1} |h_{j-1}| \nu^{j-1}  \right)+  \frac{1}{\nu}  \left(2 \sum_{j \ge 1} |h_{j+1}| \nu^{j+1}  \right)
\\
& = \nu \left( |h_0| + \| h\|_\nu \right) + \frac{1}{\nu} \left( \|h\|_\nu - 2|h_1| - |h_0| \right) 
\\
& = \left( \nu + \frac{1}{\nu} \right) \| h\|_\nu + \left( \nu - \frac{1}{\nu} \right) |h_0| - \frac{2}{\nu} |h_1|
\\ 
& \le \left( \nu + \frac{1}{\nu} \right) \| h\|_\nu + \left( \nu - \frac{1}{\nu} \right) \| h\|_\nu 
= (2 \nu ) \| h\|_\nu.
\end{align*}
Fix $r^* \ge r$ (a condition that needs to be checked a posteriori), and set
\begin{align*}
z^{(2)}_{a_2} & \bydef  4 \|\ba_3\|_\nu  + 2 |\bar \lambda|  + 3  r^* ,
\\
z^{(2)}_{a_3}  & \bydef 4 \|\ba_3\|_\nu  + 2 \| \ba_2 \|_\nu + 3 r^* ,
\\
z^{(2)}_{b_2}  & \bydef 4 \| \bb_3\|_\nu + 4 \|\ba_3\|_\nu + 4 |\bar \lambda|  + 6 r^* ,
\\
z^{(2)}_{b_3}  & \bydef 4 \|\ba_2\|_\nu + 2 \|\bb_2\|_\nu  + 4 \|\bb_3\|_\nu  + 8 \|\ba_3\|_\nu + 9 r^*.
\end{align*}
Under these assumptions, we can verify that
\[
\| A[DF(\bx+z) - D F(\bx)]\|_{B(X)} \le \|A\|_{B(X)} \| T \|_{B(\ell_\nu^1)} 
\max_{\alpha \in \{a_2,a_3,b_2,b_3\} }\left\{ \|\hat w_{\alpha}\|_\nu \right\}.
\]
Therefore, we can set
\begin{equation} \label{eq:Z2_sn}
Z_2 \bydef 2 \nu  \|A\|_{B(X)} \max \left( z^{(2)}_{a_2},z^{(2)}_{a_3},z^{(2)}_{b_2},z^{(2)}_{b_3} \right),
\end{equation}
where the computation of $\|A\|_{B(X)}$ is obtained with the same approach in Section~\ref{sec:Z0}. 

\subsection{Proof of Theorem  \ref{thm:saddle-node}} \label{sec:proof_sn}

We fix $m=65$ and obtain (using Newton's method) a numerical approximation $\bx \in \R^{6m+3}=\R^{393}$ such that $F^{(m)}(\bx) \approx 0$.
We fixed $\nu = 1.05$, and combining the explicit and computable bounds $Y_0$, $Z_0$, $Z_1$ and $Z_2$ given respectively by
\eqref{eq:Y0_sn}, \eqref{eq:Z0_sn}, \eqref{eq:Z1_sn} and \eqref{eq:Z2_sn}, we defined the radii polynomial 
$p(r)$ as in \eqref{eq:radii_polynomial} and applied the radii polynomial approach of Theorem~\ref{thm:radPolyBanach} to show that $p(r_0)<0$ with $r_0 = 5.7 \times 10^{-12}$. This yields the existence of a unique $\tx \in B_{r_0}(\bx) \subset X$ such that $F(\tx)=0$. This rigorous error bound implies the proof of Theorem~\ref{thm:saddle-node}. The graph of the solution is portrayed in Figure~\ref{fig:saddle_node_steady_state}.

Choosing the values for $m$ and $\nu$ is heuristic, non unique and done essentially so that the bound $Z_1$ satisfies $Z_1<1$ (this is indeed  a necessary condition for \eqref{eq:p(r0)<0} to hold for some $r_0>0$). Recalling \eqref{eq:Z1_sn}, and the definition $\hat z_\delta = \hat z_\gamma = \frac{1}{\nu}$, it is clear that $Z_1<1$ only if $\nu>1$. However, it cannot be taken too large as the solution itself may not have enough regularity to be in the space $\ell_\nu^1$. Also, taking $\nu>1$ provides decay in the bound $\Psi_k(\bar \alpha)$ in \eqref{eq:bounds_lkac}. Finally the choice of $m$ needs to be large enough so that the {\em defect bound} $Y_0$ is small enough and so that the {\em tail terms} $\frac{1}{2m} \| T\|_{B(\ell_\nu^1)} \hat \psi_{\alpha}^{(\infty)}$ in the definition of $z^{(1)}_\alpha$ are less than $1$. While other choices would have worked, we found that the choice $m=65$ and $\nu = 1.05$ yielded the best rigorous error bound possible in that of $r_0 = 5.7 \times 10^{-12}$.

\section{Conclusion}

In this paper, we studied a particular one-dimensional model for a 
microelectrical device in equation \eqref{1}. This equation forms part of a wide
range of Hamiltonian PDEs modeling physical phenomena such as the nonlinear
wave equation, the nonlinear Schr\"odinger equation, beam's equation, and Euler's
equation and its multiple approximations appearing in water waves. All these
Hamiltonian PDEs exhibit trivial or steady solutions, and near these steady
solutions there are periodic and quasiperiodic solutions. However, proving
existence of such solutions exhibits a small divisor problem unless one
imposes a special relation between the period and the domain of the equation.
In that case the equation can be solved with a Lyapunov-Schmidt procedure by
separating the equation into the kernel and range equation. Even if the small
divisor problem can be avoided, there are other mathematical difficulties
associated to the existence of periodic solutions such as the lack of
compactness of the linearized operator for the range equation or the infinite
dimension of the kernel equation.

Equation \eqref{1} has a family of stable steady states $\left(  u_{\lambda
},\lambda\right)  $ for $\lambda\in\lbrack0,\lambda_{\ast}]$. 
In this paper, we proved the existence of an infinite number of continuous branches
of periodic solutions arising from the steady solution $u_{\lambda}$. The
local branches have fixed periods satisfying a rational relation with the
space length and arise from $u_{\lambda}$ for bifurcation values
$\lambda \in\lbrack0,\lambda_{\ast}]$. In order to tackle the difficulties
associated to prove this fact we used a combination of analytic estimates and
computer-assisted proofs. We also introduced a systematic setting to compute
numerically these branches of solutions. 

The specific features of equation \eqref{1} imply that the set of bifurcation
values $\lambda$ is not only infinite, but actually it is a dense subset
of parameters in $[0,\lambda_{\ast}]$. Moreover, in the complement of the
dense set of $[0,\lambda_{\ast}]$, KAM theory for Hamiltonian PDEs may be
used to prove the existence of periodic solutions for a subset of almost full
measure of $[0,\lambda_{\ast}]$. However, in such case the periods must
satisfy some Diophantine relations and the periodic solutions form cantor-like
sets, that is they do not form continuous families as in our results.

We finish our conclusion by mentioning some of the possible extensions
of our work:

\begin{enumerate}
\item  The methods presented here can be used to solve similar problems in
other Hamiltonian PDEs. In particular, the same procedure can be implemented
to prove the existence of periodic solutions in the equation of MEMS with
other dielectric permitivity properties $f(x)\neq1$ or for other nonlinear
wave equations of the form $u_{tt}-u_{xx}=$ $f(x,u)$.

\item In order to prove our result, in Section~\ref{sec:cap} we implement a
computer-assisted proof to validated the steady solution $u_{\lambda_{\ast}}$
at critical value $\lambda_{\ast}$. Actually, the numerical setting in
Sections~\ref{sec:cont_steady-states} and~\ref{sec:cont_eigenpairs}, and the methods of Section~\ref{sec:cap} can be used to
validate also the trivial branch $u_{\lambda}$ and its spectrum
(eigenfunctions and eigenvalues). Furthermore, similar procedures can be used
to validate also radial steady solutions for the nonlinear equation in more
dimensions
\[
\Delta u=\lambda f(u),\qquad u_{\partial B}=0,
\]
with analytic nonlinearities such as $f(u)=u^{p}$ or $f(u)=e^{u}$ (e.g. see \cite{MR3917433}).

\item Unfortunately, there are no readily available methods to validate numerically the
periodic solutions obtained in Section~\ref{sec:cont_periodic_solutions}. The problem is that the inverse of
the hyperbolic operator $L$ is only bounded but not compact. Indeed, the
methods of Section~\ref{sec:cap} depend strongly on the compactness of the inverse operators
to validated the numerical solutions, because the compactness allows to obtain
estimates for the Galerkin approximation or truncation of the linear operators.
Further research and new ideas are required to validate rigorously the
numerical computations of the periodic solutions.
\end{enumerate}

\section{Acknowledgments} CGA is indebted to G. Flores and M.
Tejada-Wriedt for discussions related to this project.


\end{document}